\documentclass[a4paper]{article}
\usepackage{amsfonts}
\usepackage{amssymb}
\usepackage{amsmath}
\usepackage{amsthm}
\usepackage{graphicx}
\usepackage{subcaption}
\usepackage{multicol}
\usepackage{siunitx}
\usepackage{url}
\usepackage[utf8]{inputenc}
\usepackage[nottoc]{tocbibind}
\usepackage{tikz}
\usepackage[symbol]{footmisc}
\usepackage{xcolor}
\usepackage{mleftright}
\usepackage{verbatim}
\usepackage[ruled,vlined]{algorithm2e}
\usepackage[noend]{algpseudocode}
\usepackage{lipsum}
\usepackage{mathrsfs}
\usepackage{tikz-cd}
\usepackage{enumitem}

\usepackage{dsfont}
\usepackage{hyperref}
\hypersetup{
	colorlinks=true,
	linkcolor=blue,
	filecolor=magenta,      
	urlcolor=cyan,
}
\usepackage{tabularx}
\usepackage{multirow}
\usepackage{longtable}
\usepackage{emptypage}
\usepackage{cleveref}
\usepackage{enumitem}
\usepackage[toc,page]{appendix}
\usepackage{mathtools}

\usepackage{a4wide}

\newtheorem{lemma}{Lemma}[section]

\newtheorem*{theorem*}{Theorem}
\newtheorem{definition}[lemma]{Definition}
\newtheorem{proposition}[lemma]{Proposition}
\newtheorem*{proposition*}{Proposition}
\newtheorem{example}[lemma]{Example}
\newtheorem{corollary}[lemma]{Corollary}
\newtheorem{remark}[lemma]{Remark}

\newtheorem{notation}[lemma]{Notation}
\newtheorem{question}{Question}
\newtheorem{result}{Theorem}

\crefalias{result}{theorem}

\DeclareMathOperator{\lcm}{lcm}

\newcommand{\QQ}{\mathbb{Q}}
\newcommand{\RR}{\mathbb{R}}

\newcommand{\ZZ}{\mathbb{Z}}

\newcommand{\NN}{\mathbb{N}}

\newcommand{\inv}{^{-1}}
\newcommand{\comp}{\circ}

\newcommand{\abs}[1]{\left|#1\right|}
\newcommand{\norm}[1]{\left|\left|#1\right|\right|}
\newcommand{\set}[1]{\left\{#1\right\}}

\newcommand{\cat}[1]{\mathscr{#1}}

\newcommand{\mf}[1]{\mathfrak{#1}}
\newcommand{\mc}[1]{\mathcal{#1}}

\newcommand{\without}[1]{\backslash\{#1\}} 

\newcommand{\normal}{\vartriangleleft}
\newcommand{\conj}{\mathrm{Conj}}
\newcommand{\rf}{\mathrm{Rf}}
\newcommand{\cl}{\mathrm{Cl}}

\newcommand{\lelangle}{\left\langle}
\newcommand{\rerangle}{\right\rangle}
\newcommand{\llangle}{\langle\hspace{-2.5pt}\langle}
\newcommand{\rrangle}{\rangle\hspace{-2.5pt}\rangle}

\newcommand{\change}[1]{}

\title{Residual finiteness properties of some of Halls groups.}
\author{Lukas Vandeputte\thanks{The author  kindly acknowledges the support by the group of science Engineering and Technology at KU Leuven Campus Kulak.}}
\begin{document}
	\maketitle
	\begin{abstract}
		In this article we study a class of central extensions of $\ZZ\wr\ZZ$, as first described by Hall.
		On the one hand, we consider groups of this type with cyclic centre, our construction yields a rich class of groups. In particular we obtain a group that is conjugacy separable with solvable word problem but unsolvable conjugacy problem, we obtain a group with large conjugacy separability growth but small conjugator length function and residual finiteness growth, and we also obtain both a class of groups that for most functions $f:\NN\rightarrow\NN$ larger then $n^3$, contain a group $G$ such that the conjugator length of $G$ is given by $f$, as well as a group where the conjugator length is superlinear but subquadratic.
		
		On the other hand we also consider a different group with larger centre. This is the first example of a group where the residual finiteness growth is faster than any polynomial but slower than any exponential.
	\end{abstract}
	\section{Introduction}
	We study a class of central extensions of $\ZZ\wr\ZZ$. These groups where first constructed by Hall in \cite{hall1954finiteness} to demonstrate that a certain variety of $2$-generated groups has uncountably many non-isomorphic groups. These groups are all central quotients of a $2$-generated $3$-step solvable group $G_0$ with centre $\bigoplus_\NN\ZZ$. This large central subgroup allows us to encode a lot of information by taking a central quotient.
	
	We mainly study these groups through the lens of residual-finiteness and conjugacy-separability.
	Remember that a group $G$ is \textbf{residually finite} if for any non-trivial element $g\in G$, there exists a finite index normal subgroup $N$ such that $gN\neq N$ and that a group is \textbf{conjugacy separable} if for any pair of non-conjugate elements $g_1,g_2\in G$, there exists a finite index normal subgroup $N$ such that $g_1N$ and $g_2N$ are non-conjugate in $\frac{G}{N}$. A classical way to motivate these notions, is that if a finitely-presented group is residually finite or conjugacy separable, then that group has solvable word problem or solvable conjugacy problem respectively. This is done by McKinsey's algorithm, which enumerates all finite quotients and checks triviality of an element or conjugacy of these elements in those quotients.
	
	Notice that the condition of being finitely presented guarantees that it is possible to enumerate these quotients, and thus for these algorithms to work. For instance in \cite{meskin1974finitely} the author constructs a finitely generated recursively presented residually finite group with unsolvable word problem.
	
	A first class of quotients we consider are those with cyclic centre. These groups are then parametrised by a function $f:\NN\rightarrow \NN$. We characteristics of our groups are then determined by the growth and number theory of this sequence, leading to a rich class of groups. In particular we obtain criteria for when these groups are residually finite and conjugacy separable. In particular we obtain.
	\begin{result}
		There exist a finitely generated, recursively presented, conjugacy separable group with solvable word problem but unsolvable conjugacy problem.
	\end{result}
	
	This answers a question by Ashot Minasyan in the negative \cite{minasyan2010conjugacy}.
	Note that a very recent preprint by Darbinyan and Rauzy answers the same question \cite{darbinyan2025robust}. Our construction is different and our results were found independently.
	
	A more recent angle to residual finiteness comes through the lens of effective separability. For a residually finite group $G$ that is finitely generated by $S$, in \cite{bou2010quantifying} the author constructs the \textbf{residual finiteness growth function} $\rf_{G,S}:\NN\rightarrow\NN$ which maps an integer $n$ to $$
	\rf_{G,S}(n)=\max_{\substack{g\in G\without 1\\\norm{g}_S\leq n}}\left(\min\{Q\mid \varphi:G\rightarrow Q,\abs{Q}<\infty,\varphi(g)\neq 1\}\right).
	$$
	This function is known for a number of groups. A recent survey on this can be found in \cite{dere2022survey}. For linear groups, among others, it is known that the residual finiteness growth is bounded above by a polynomial. On the other hand by \cite{bradford2024spectrum} it is known that most reasonable functions larger than the exponential can up to a natural equivalence relation be obtained as $\rf_{G,S}$ for some finitely generated group $G$.
	There is also a conjugacy analogue to $\rf$, given by the \textbf{conjugacy separability function} $\conj_{G,S}:\NN\rightarrow\NN$ which maps an integer $n$ to$$
	\conj_{G,S}(n)=\max_{\substack{g_1,g_2\in G\\\norm{g_1}_S,\norm{g_2}_S\leq n\\g_1\not\sim g_2}}\left(\min\{Q\mid \varphi:G\rightarrow Q,\abs{Q}<\infty,\varphi(g_1)\not\sim \varphi(g_2)\}\right).
	$$
	The knowledge on $\conj$ is more limited. Some results are again present in \cite{dere2022survey}. There is also an analogue of \cite{bradford2024spectrum} given in \cite{vandeputte2024spectrum}: given two functions $f_1\leq f_2$ that grow sufficiently fast, then there exists some group $G$ such that $\rf_{G,S}$ is given by $f_1$ and $\conj_{G,S}$ by $f_2$.
	
	These functions give rise to bounds to the time complexity of one half of McKinsey's algorithm. The other side of McKinsey's algorithm for the conjugacy problem works by enumerating all possible conjugates of one of the two group elements. The complexity of this side of the algorithm is determined by the \textbf{conjugator length function} $\cl_{G,S}:\NN\rightarrow\NN$ which maps an integer $n$ to$$
	\cl_{G,S}(n)=\max_{\substack{g_1,g_2\in G\\
	\norm{g_1}_S,\norm{g_2}_S\leq n\\
	g_1\sim g_2}}(\min\{\norm{h}_S\mid h\in G,hg_1h\inv=g_2\})
	$$
	
	Also this function has been studied on its own. A lot of groups have been shown to have at most linear conjugator length function such as word hyperbolic groups \cite{lysenok1989some}, mapping class groups\cite{tau2013linearly}, lamplighers\cite{sale2016conjugacy}, free-by-cyclic groups\cite{bridson2025conjugacy}, affine coxeter groups and split crystalographic groups\cite{rego2025conjugator}, but other functions appear as well, for instance all polynomials appear in the nilpotent groups\cite{bridson2025lengths}. It is also not difficult to construct groups with with larger conjugator length function: for instance if a group $G$ has solvable word problem but unsolvable conjugacy problem, then the $\cl_{G,S}$ grows faster than any computable function.
	
	One might expect that knowing $\rf$ and $\cl$ gives insight on $\conj$. For instance one might hope that given $g_1,g_2\in G$ non-conjugate of length $n$, if for all $h$ of length at most $\cl_{G,S}(n)$ one can separate $hg_1h\inv$ from $g_2$ in a finite quotient, this aids in finding a finite quotient separating the conjugacy classes of $g_1$ and $g_2$. Nevertheless we show that this is not the case. Indeed we construct a class of groups similar to the one of \Cref{prop:ExistsConjNotSolvable} that gives us the following:
	\begin{result}
			There exists a family of finitely generated recursively presented torsion-free 3-step solvable conjugacy separable groups $G_d$ such that:\begin{itemize}
				\item$\rf_{G_d}\prec \exp(n^\delta)$ for some integer $\delta$; 
				\item$\cl_{G_d}$ is polynomial; \item the conjugacy problem is polynomial time solvable; \item for any computable function $\Phi$, there exists some $d$ such that $\conj_{G_d}\succ \Phi$.
			\end{itemize}
	\end{result}
	
	Notice that in particular, this result gives a partial answer to \cite[Question 1.]{vandeputte2024spectrum} which asks about $\rf$ and $\conj$ in torsion-free groups. Notice furthermore that this result is a strengthening of \cite{vandeputte2025independence}. 
	
	We further demonstrate the control we have over the conjugator length with our construction by proving the following:
	\begin{result}
		Let $\Phi$ be a function such that either one of the following holds:
		\begin{enumerate}[label=(\alph*)]
			\item $\forall i\in \NN, \Phi(i)\geq i^3$ and $\Phi(3i)\geq 3\Phi(i)$;\\
			\item $\Phi(i)\geq (i)^{1+\log_3(2)+\log_3\log_3(i)}$.\\
		\end{enumerate}		
		Then there exists a torsion-free, finitely generated, 3-step solvable group $G$ such that $\cl_G\simeq \Phi$. Furthermore if $\Phi$ is computable then $G$ can be chosen to be recursively presented.
	\end{result}
	To our knowledge, we are the first to demonstrate such a large class of functions appears as $\cl_G$.
	
	Not only large functions appear as conjugator length, however. In \cite[Problem 1]{gillis2026conjugator}
	the authors ask if there are finitely generated groups $G$ with $\cl_G(n)$ strictly in between $n$ and $n^2$.
	We answer this in the positive in \Cref{prop:clIntermediateLower} and \Cref{prop:clIntermediateUpper} by constructing a group $G$ with $n^{\frac{4}{3}}\prec \cl_G(n)\prec n^{\frac{5}{3}}$.
	
	note that aforementioned \cite{gillis2026conjugator} obtains similar results to \Cref{prop:clSpectrum} but for finitely presented groups. This paper was released later then the first two versions of our own (the second version containing \Cref{prop:clSpectrum}). The methods between the two are completely different.
	
	In the final section, we shift our focus back to $\rf$.
	So far, there are no known examples of finitely generated groups where $\rf_{G,S}$ lies strictly in between the polynomials and the exponential functions.	
	By constructing another central quotient of $G_0$, we obtain\begin{result}
		There exists a finitely generated recursively presented group $G$ such that for any polynomial $f$, $\rf_{G,S}(n)\geq f(n)$ for $n$ sufficiently large, and such that $\rf_{G,S}\leq\exp(n^{\frac{1}{2}+\epsilon})$ for $n$ sufficiently large.
	\end{result}
		
	In doing so we also demonstrate:
	\begin{result}
		There exists some constant $c>0$ such that $$
		\rf_{\ZZ\wr\ZZ,S}(n)\leq cn^2.
		$$
	\end{result}
	This matches the best known upper bound for $\frac{\ZZ}{n\ZZ}\wr\ZZ$ and answers \cite[Question 3.25.]{dere2022survey}.

	The remainder of this article is structured as follows: In \cref{sec:hall} we start by describing the general class of groups, as well as some general results of these groups. Here we also demonstrate our \Cref{prop:rflamplighter}. In \cref{sec:conjugacy}, we describe in more detail the subclass of groups with cyclic centre. We specify further in $3$ different ways to find a group satisfying \Cref{prop:ExistsConjNotSolvable}, a class of groups satisfying \Cref{prop:cLengthCojSep}, and a class of groups satisfying \Cref{prop:clSpectrum}.
	In \cref{sec:intermediate}, we describe a different group, this time with larger centre, and demonstrate it satisfies \Cref{prop:ExistsIntermediate}.

	\section{Preliminaries}
	We start by fixing the convention that given $g_1,g_2\in G$, $[g_1,g_2]=g_1g_2g_1\inv g_2\inv.$
	A group is \textbf{residually finite} if for any non-trivial element $g$, there exists a finite quotient ${\varphi:G\rightarrow Q}$ such that $\varphi(g)\neq1$. Related to this a group is called \textbf{conjugacy separable} if for any pair of non-conjugate elements $g_1,g_2$, there exists a finite quotient such that $\varphi(g_1)$ and $\varphi(g_2)$ are non-conjugate. The word problem in a finitely presented residually finite group is solvable\cite{Malcev1985}, as is the conjugacy problem in a finitely presented conjugacy separable group. 
	A lot of groups have been shown to be residually finite or conjugacy separable. Important for our purposes is that finitely generated Linear groups (in particular $\ZZ\wr\ZZ$) have been shown to be residually finite\cite{malcev1940isomorphic}, and that finitely generated virtually nilpotent groups and $\ZZ\wr\ZZ$ are known to be conjugacy separable\cite{formanek1976conjugate,remeslennikov1971finite}.
	
	The degree to which a group $G$, finitely generated by $S$, is residually finite can be measured through the function$$ \rf_{G,S}(n)=\max_{\substack{g\in G\without 1\\\norm{g}_S\leq n}}\left(\min\{Q\mid \varphi:G\rightarrow Q,\abs{Q}<\infty,\varphi(g)\neq 1\}\right).$$
	
	Here $\norm{\_}_S$ denotes the word norm on $G$, defined as $$
	\norm{g}_S=\min\{k\in\NN\mid \exists s_1,s_2,\cdots s_k\in S\cup S\inv:\:s_1s_2\cdots s_k=g\}
	.$$ The dependence of $\rf_{G,S}$ on $S$ is only up to a constant factor:
	\begin{lemma}
		Let $G$ be a residually finite group and let $S$ and $T$ be finite generating sets of $G$. Then there exists a constant $C$ such that$$
		\rf_{G,S}(n)\leq C\rf_{G,T}(Cn)
		$$
		and vice versa.
	\end{lemma}
	\begin{proof}
		See for instance \cite[Lemma 1.1]{bou2010quantifying}.
	\end{proof}
	When $f,g:\NN_{>0}\rightarrow\NN_{>0}$ are increasing functions, we denote $f\succ g$ if $f(n)\geq Cg(Cn)$ for some constant $C$, and we denote $f\simeq g$ if $f\prec g$ and $f\succ g$.
	The above lemma thus gives us that $\rf_{G,S}$ is up to $\simeq$ independent of $S$. We may thus denote $\rf_{G}$.
	
	It has been shown in \cite{bou2010quantifying} that $\rf_{\ZZ}\simeq \ln(n)$. The `$\prec$' implied in this statement follows from the $K=\QQ$ case of the following statement.	
	
	\begin{proposition}\label{prop:RfZChebotarev}
		Let $\mc O_K$ be the ring of integers of a Galois number field. There exists some constant $C$ such that whenever $x\in \mc O_K$ is non-trivial, then there exists an ideal $I\normal\mc O_K$ of prime norm $p$ such that $p\leq C\log(C\norm{x}_K)$ where $\norm{x}_K=[\mc O_K:x\mc O_K]$ denotes the field norm of $x$ in $K$.
	\end{proposition}
	\begin{proof}
		By Chebotarevs density theorem as stated in \cite[Theorem 3.4.]{serre2012lectures}, the set of primes $p$ less than $n$ such that there exists some ideal in $\mc O_K$ of norm $p$, is asymptotically of size $\Theta(\frac{n}{\log(n)})$. It follows for instance by \cite[Proposition 4.7.]{dere2024residual} that for any $y\in\ZZ_{>0}$, there exists such a prime $p$ such that $p\nmid y$ and $p\leq C\log(Cy)$. The result follows when $y=\norm{x}_K$ and $I$ is a prime ideal of order $p$.
	\end{proof}
	Similar to how the residual finiteness growth measures how residually finite a group is, the conjugacy separability is quantised using the \textbf{conjugacy separability growth} given by$$
	\conj_{G,S}(n)=\max_{\substack{g_1,g_2\in G\\\norm{g_1}_S,\norm{g_2}_S\leq n\\g_1\not\sim g_2}}\left(\min\{Q\mid \varphi:G\rightarrow Q,\abs{Q}<\infty,\varphi(g_1)\not\sim \varphi(g_2)\}\right).
	$$
	Again this function only depends on $S$ up to the equivalence relation $\simeq$.
	
	Similarly \textbf{conjugator length function}, given by $\cl_{G,S}:\NN\rightarrow\NN$ mapping an integer $n$ to$$
	\cl_{G,S}(n)=\max_{\substack{g_1,g_2\in G\\
			\norm{g_1}_S,\norm{g_2}_S\leq n\\
			g_1\sim g_2}}(\min\{\norm{h}_S\mid h\in G,hg_1h\inv=g_2\})
	$$
	also only depends on $S$ up to $\simeq$.

	\section{Halls group}\label{sec:hall}
	We start of by constructing the group of \cite[Theorem 7]{hall1954finiteness}.
	Let $N_0$ be the $2$-step nilpotent group with presentation$$
	\lelangle \{a_i,c_i\mid i\in\ZZ\}\bigg\vert \begin{matrix}[a_i,a_j]=c_{i-j}&i,j\in\ZZ\\ [c_i,a_j]=[c_i,c_j]=1&i,j\in\ZZ
		\end{matrix}\rerangle.
	$$
	Notice that by taking $i=j$, it follows that $c_0$ is the identity. By swapping $i$ and $j$, we obtain that $c_i$ is the inverse of $c_{-i}$. There is a natural action of $\ZZ$ on $N$, given by $\rho_i(a_j)=a_{i+j}$ and $\rho_i(c_j)=c_j$. One checks easily that this action is well defined. We can thus consider $G_0$ the semi direct product $N_0\rtimes_\rho \ZZ$. 
	
	The presentation associated to this semi direct product is then given by 
	$$
	G_0=\lelangle \{t,a_i,c_i\vert i\in\ZZ\}\Bigg\vert \begin{matrix} ta_it\inv=a_{i+1}&i\in\ZZ\\
		tc_it\inv=c_i&i\in\ZZ\\
		[a_i,a_j]=c_{i-j}&i,j\in\ZZ\\ [c_i,a_j]=[c_i,c_j]=1&i,j\in\ZZ
	\end{matrix}\rerangle.
	$$
	Using relations of the first and third kind, one sees that this group is generated by $S=\{t,a_0\}$. 
	Let $C_0$ be the subgroup generated by the generators $c_i$, notice that $C_0$ is not only central in $D_0$, but also in $G_0$.
	Notice that $C_0$ is the free abelian group on $\{c_i\mid i\in \ZZ_{>0}\}$.
	
	The groups we consider in this paper are obtained by taking quotients of $G_0$, with kernel in $C_0$. The fact that $C_0$ is a large central subgroup, gives us a lot of flexibility to encode information into this quotient. We will see that depending on which quotient $G$ of $G_0$ we take, the properties of $G_0$ change drastically.
	\begin{definition}
		Let $G_0,C_0$ be as above, denote then $\cat H$ the class of central quotients of $G_0$, that is:$$
		\cat H=\left\{\frac{G_0}{K}\big\vert K\normal C_0\right\}
		$$ 
	\end{definition}
	Throughout the paper, we will use $t,a_i,c_i$ as elements of $G$ whenever $G\in\cat H$. These are then to be regarded as the images of the elements of $G_0$ with the same name. When confusion is possible, we clarify to which group they belong by denoting $a_i\in G$ or $a_i\in G_0$.
	\begin{notation}
		Let $G\in \cat H$ be a central quotient of $G_0$, denote then $N$ and $C$ as the images of $N_0$ and $C_0$ respectively in that quotient.
	\end{notation}
	\begin{example}
		The wreath product $\ZZ\wr\ZZ\in \cat H$. Here $C$ is the trivial group and $N=\bigoplus_{\ZZ}\ZZ$. 
		Notice that $\ZZ\wr\ZZ$ is  a quotient of all other groups $G\in \cat H$.
	\end{example}
	
	The remainder of this section will be spend on general properties of $G_0$ or $\cat H$, then in \cref{sec:conjugacy} and \cref{sec:intermediate} we will specialize to specific quotients to demonstrate \Cref{prop:ExistsConjNotSolvable} and \Cref{prop:ExistsIntermediate} respectively.
	
	We begin with studying the word norm on $G_0$.
	\begin{definition}
		Let $S_n\subset N_0$ be defined as$$
		S_{n}=\set{a_i\mid i\in [-n,n]}.
		$$
	\end{definition}
	None of the sets $S_n$ generate $N_0$, nevertheless we have the following.
	\begin{lemma}\label{prop:support}
		Let $g\in N_0<G_0$ and suppose that $\norm{g}_S\leq n$, then $g\in \langle S_{n}\rangle$ and $\norm{g}_{S_n}\leq n$.
	\end{lemma}
	\begin{proof}
		We may rewrite $g=t^{n_1}a^{m_1}t^{n_2}a^{m_2}t^{n_3}\cdots t^{n_k}a^{m_k}t^{n_{k+1}}$ such that $n_{k+1}+\sum_{i=1}^k \abs{n_i}+\abs{m_i}\leq n$. We may rewrite this as$$
		t^{n_1}a^{m_1}t^{-n_1}\cdot t^{n_1+n_2}a^{m_2}t^{-n_1-n_2}\cdots t^{n_1+\cdots+n_k}a^{m_k}t^{-n_1-\cdots-n_k}\cdot t^{n_1+\cdots+n_k+n_{k+1}}
		$$
		As $g\in N$, we have that $n_1+\cdots+n_k+n_{k+1}=0$, the above is thus equal to$$
		a_{n_1}^{m_1}a_{n_1+n_2}^{m_2}\cdots a_{n_1+\cdots+n_k}^{m_k}
		$$\change{n and m indices needed to be swapped}
		The partial sums $n_1+n_2+\cdots n_i$ are in absolute value bounded above by $n$ thus $g\in \langle S_{n}\rangle$. Furthermore $\sum \abs{m_i}$ is also bounded by $m$ thus $\norm{g}_{S_n}\leq n$.
	\end{proof}
	In the context of lamplighter groups, there is a partial converse to this.
	\begin{lemma}\label{prop:supportInverse}
		Let $g\in \langle S_n\rangle<\ZZ\wr\ZZ$, and suppose $\norm{g}_{S_n}\leq m$, then $\norm{g}_S\leq 4n+m$.
	\end{lemma}
	\begin{proof}
		Let $g$ be given by $\sum_{i=-n}^na_i^{e_i}$ such that $\sum \abs{e_i}\leq m$, then $g$ is also given by $$
		g=t^{-m}\left(\prod_{i=-m}^{m-1}a_0^{e_i}t\right)a_0^{e_m}t^{-m}.
		$$
	\end{proof}
	Note that the above might fail if $G\neq \ZZ\wr\ZZ$.
	We have an analogous result to \Cref{prop:support} for the subgroup $C$.
	\begin{definition}
		Let $S_{c,n}\subset C_0$ be defined as$$
		S_{c,n}=\set{c_i\mid i\in [1,n]}.
		$$
	\end{definition}
	\begin{corollary}\label{prop:supportC}
		Let $g\in \langle \set{c_i\mid i\in \ZZ}\rangle$ and suppose that $\norm{g}\leq n$, then $g\in \lelangle S_{c,n}\rerangle$ and $\norm{g}_{S_{c,n}}<n^2$.
	\end{corollary}
	\begin{proof}
		Notice that if $g\in \langle \set{c_i\mid i\in \ZZ}\rangle$ then we also have $g\in\langle \set{a_i\mid i\in \ZZ}\rangle$. The first part of the statement follows by \Cref{prop:support} as $\lelangle S_n\rerangle\cap C_0=\lelangle S_{c,n}\rerangle$. The second part follows by well known results on the word norm on finitely generated nilpotent groups, as can be found for instance in the proof of \cite[Theorem II.1]{guivarch1973croissance}.
	\end{proof}
	Note that \Cref{prop:support} and \Cref{prop:supportC} are both preserved under taking central quotients of $G_0$. These statements thus still hold after replacing $G_0,N_0$ and $C_0$ with $G,N$ and $C$ respectively for some $G\in\cat H$.

	To talk about residual finiteness of central quotients of $G_0$, we often need to reduce to the residual finiteness of $\ZZ\wr\ZZ$. This group is linear and thus known to be residually finite. We provide an upper bound on the residual finiteness growth. Notice that $\ZZ\wr\ZZ$ is a central quotient of $G_0$ where the kernel is the entire centre. We thus denote $a_i$ the projections of the generators $a_i\in G_0$.
	\setcounter{result}{4}
	\begin{result}\label{prop:rflamplighter}
		$\rf_{\ZZ\wr\ZZ}\prec n^2$.
	\end{result}
	\begin{proof}
		Fix $K$ some real degree $2$ number field and let $\mc O_K$ be its ring of integers. (For instance let $K=\QQ[\sqrt{2}]$ and $\mc O_K=\ZZ[\sqrt{2}]$). Fix $u$ a unit in $\mc O_K$ of infinite order.
		
		Let $g\in \ZZ\wr\ZZ\without 1$ with $\norm{n}_S\leq n$. By \cite[Theorem 2.2.]{bou2010quantifying}, we know that $\rf_{\ZZ}\prec \log(n)$, we may thus assume that $g$ lies in $\lelangle\{a_i\mid i\in\ZZ\}\rerangle$ and thus by \Cref{prop:support}, we have that $g\in \langle S_{n}\rangle$ and that $\norm{g}_{S_{n}}\leq n$. We may thus write $g=\sum_{\abs{i}\leq n} \gamma_i a_i$ with $\abs{\gamma_i}\leq n$. Denote $f(X)$ the polynomial $\sum \gamma_iX^{i+n}$ of degree $2n+1$. As this polynomial is not identically $0$, we can thus find some $k$ with $\abs{k}\leq n+1$ such that $f(u^k)\neq 0$. We can bound the complex norm $\abs{f(u^k)}$ by $(2n+1)n \abs{u}^{n+1}$ which is bounded above by $\exp(cn)$ for some constant $c$, independent from $n$. Similarly, if we denote $\overline u$ the Galois conjugate of $u$, then we again obtain a bound $\abs{f(\overline u^k)}\leq \exp(c'n)$. The field norm $\norm{f(u^k)}_{\mc O_k}$ is the integer given by $f(u^k)f(\overline u^k)\leq \exp((c+c')n)$.
			
		By \Cref{prop:RfZChebotarev}, we might find some prime ideal $I$ of $\mc O_k$ such that $f(u^k)\notin I$ and such that the index of $I$ is at most $c''\log(\norm{f(u^k)}_{\mc O_k})\leq c''(c+c')n$.
		Let $r$ be the multiplicative order of $u$ in $\frac{\mc O_K}{I}$. Notice that $r$ is bounded above by $[\mc O_K:I]$.
		Let $\frac{\mc O_K}{I}\rtimes\frac{\ZZ}{r\ZZ}$ denote the semi direct product where the right group acts on the left by multiplication with $u^k$ and consider the morphism $\pi:\ZZ\wr\ZZ\rightarrow \frac{\mc O_K}{I}\rtimes\frac{\ZZ}{r\ZZ}$
		Where $(0,n)$ gets mapped to $(0,n+r\ZZ)$ and where $a_i$ gets mapped to $(u^{ki},0)$.
		Notice that the image of $g$ under this morphism is given by $(u^{-kn}f(u^k)+I,0)$, which is a non-trivial element. The result follows as the order of $\frac{\mc O_K}{I}\rtimes\frac{\ZZ}{r\ZZ}$ is bounded by $(c''(c+c'))^2n^2$.
	\end{proof}

	To study conjugacy in the groups in $\cat H$, we need to study centralisers in $\ZZ\wr\ZZ$.
	\begin{lemma}\label{prop:Centraliser}
		Let $g\in \ZZ\wr\ZZ\backslash N$, then $g$ has cyclic centraliser. On the other hand if $g\in N$, then $g$ is centralised by  $N$.
	\end{lemma}
	\begin{proof}
		The first case follows for instance from \cite[Corollary 4]{Meldrum1979centralisers}. If $g\in N$, then clearly every other element of $N$ commutes with $g$. On the other hand, again by \cite[Corollary 4]{Meldrum1979centralisers}, no other element commutes with $g$.
	\end{proof}
	\begin{corollary}\label{prop:conjInSubgroup}
		Let $g_1,g_2\in G$ such that $g_1C=g_2C$. If $g_1,g_2\notin N$, then $g_1\sim g_2$ if and only if $g_1=g_2$. On the other hand if $g_1,g_1\in N$, then $g_1\sim g_2$ if and only if $g_2g_1\inv\in [g_1,N]$.
	\end{corollary}
	\begin{proof}
		For the first case, Let $g_0$ be such that $g_0C$ generates $Z(g_1C)$. As the centraliser of $g_1C$ is cyclic by \Cref{prop:Centraliser} and as $g_1C$ centralises itself, there exists some $k\in\ZZ$ and $c\in C$ such that $g_1=g_0^kc$. As $c$ is central, it follows that $g_0$ and $g_1$ commute. Suppose that there exists some $h$ such that $hg_1h\inv=g_2$, then as $g_1C=g_2C$, $hC$ must centralise $g_1C$. It follows that $h$ must lie in the subgroup generated by $C$ and $g_0$. In particular, as both of these elements centralise $g_1$, it follows that $h$ itself centralises $g_1$ and thus that $g_2=hg_1h\inv=g_1$.
		
		For the second case, assume $g_1\in N$ and let $h$ be such that $hg_1h=g_2$. Again $hC$ must centralise $g_1C$ and thus by \Cref{prop:Centraliser}, $h\in N$. Rewriting we obtain $hg_1h\inv g_1\inv=g_2g_1\inv$ and thus $g_2g_1\inv\in[N,g_1]$, or equivalently as $[N,g_1]$ is a subgroup, $g_2g_1\inv\in[g_1,N]$.
	\end{proof}
	\begin{lemma}\label{prop:describeCentralizer}
		Let $g=(h,n)\in \ZZ\wr\frac{\ZZ}{I\ZZ}$, then the centraliser $Z(g)$ is generated by 
		$$P_n=\{
		a_ia_{i+n}a_{i+2n}\cdots a_{i-n}\mid i\in \frac{\ZZ}{I\ZZ}
		\}$$
		and by $g'=(h',n')$ where $in'=n$ for some $i\in\ZZ$ and where $(g')^ig\inv\in P_{n'}$.
	\end{lemma}
		\begin{proof}
		If $(h',0)\in Z(g)$, then it is obvious that $(h',0)$ lies in the subgroup generated by $P_n$.
		Let $\pi$ denote the natural projection onto $\frac{\ZZ}{I\ZZ}$. Notice that $\pi(Z(g))$ is a subgroup of $\frac{\ZZ}{I\ZZ}$ containing ${\pi(g)=n}$. It follows that there exists some $n'\mid n$ generating $\pi(Z(g))$ and thus some ${g'=(h',n')\in Z(g)}$. Let $\tilde g=(\tilde h,\tilde n)\in Z(g)$. Multiplying with some appropriate power of $g'$ returns an element in $Z(g)\cap \ZZ^n=\langle P_n\rangle$. It thus follows that $P_n\cup \{g'\}$ generates $Z(g)$. All that is left to show is that $g'$ is of the form above. Notice that $g,g'\in Z(g')$ and thus lies $(g')^ng\inv\in Z(g')\cap \ZZ^d=\langle P_{n'}\rangle$.
	\end{proof}

	Not all groups in $\cat H$ are conjugacy separable. We give a \namecref{prop:cSepPeriodicCentre} that aids us in showing when they are.
	\begin{definition}
		Let $G\in \cat H$, we say $G$ satisfies \textbf{property $\mathbf{(P)}$} if there exists some integer $P$ such that $\forall i,j\in\ZZ, [a_i,a_j]=[a_i,a_{j+P}]$.
	\end{definition}
	Notice that if $G$ satisfies property $(P)$, then $C$ is finitely generated by $\{[a_0,a_i]\mid i\in [0,P-1]\}$.
	\begin{proposition}\label{prop:cSepPeriodicCentre}
		Let $G\in\cat H$ satisfy property $(P)$, then $G$ is conjugacy separable.
	\end{proposition}
	\begin{proof}
		Let $g_1,g_2\in G$ be non-conjugate. As $\ZZ\wr\ZZ$ is conjugacy separable, we may assume that $g_2=g_1c$ for some $c\in C$.
		Denote $Z_2<G$ the subgroup of elements $\{g\in G\mid gC\in Z(g_1C)\}$.
		Notice that $g_2\nsim g_1$ precisely when $c\notin \{[g,g_1]\mid g\in Z_2\}$. Notice that this set also forms a subgroup, which we denote $[Z_2,g_1]$.
		As $C$ is finitely generated abelian, it is subgroup-separable, that is we have a subgroup $K_0<C$ of finite index $q$, such that $c\notin [Z_2,g_1]+K_0$.
		
		As $G=\tilde N\rtimes \ZZ$, we might express elements the elements $g_1,g_2$ as $(h_1,n_1)$ and $(h_2,n_2)$. By the previous assumption, $n_1=n_2$ and $h_1C=h_2C$. Let furthermore $\delta$ be minimal such that $$h_1C\in\lelangle\{a_{-\delta},a_{1-\delta},\cdots,a_\delta\}\rerangle.$$
		Let $I=Pq(4\delta+2)\max\{1,\abs{n_t}\}$ and let $K<G$ be generated by $\{a_ia_{i+I}\inv\mid i\in \ZZ\}$. Notice that $K$ is central in $N$ and normal in $G$. We will show that $g_1K_0K$ and $g_2K_0K$ are non-conjugate.
		
		We need a case distinction depending on $n_t$.
		
		First assume $n_t$ is non-trivial. Notice that if $gg_1g\inv\in g_2K_0K$, then must $gg_1g\inv CK=g_2CK$.
		By \Cref{prop:describeCentralizer}, we may thus write $g$ as ${g'}^{j}pc$ where $c\in C, p\in P_{n_1}$ and where $g'$ is such that ${g'}^ig\inv\in P_{\frac{n_1}{i}}$ for some $i$. First notice that as $c$ is central, $cg_1c\inv=g_1$ and we may thus assume that $c=1$. Next we show that $p$ and $g_1$ commute. For this it suffices to show that $p$ commutes with both $(h_1,0)$ and $(1,n_1)$. For the first, we demonstrate that $P_{n_1}$ is central in $\frac{G}{K_0K}$. Indeed, let $a_j\in N$ and $a_ia_{i+n_1}a_{i+2n_1}\cdots a_{i+I-n_1}\in P_n$, then by bi-linearity of the commutator, we have\begin{align*}
			[a_j,a_ia_{i+n}a_{i+2n}\cdots a_{i+I-n}]K_0&=\prod_{k=0}^{Pq(4\delta+2)-1}[a_j,a_{i+kn}]K_0\\
			&=\prod_{l=0}^{P(4\delta+2)-1}\prod_{k=0}^{q-1}[a_j,a_{i+l+Pk(4\delta+2)n}]K_0\\
			&=\prod_{l=0}^{P(4\delta+2)-1}[a_j,a_{i+l}]^qK_0\\
			&=1.
		\end{align*}
		For the second, observe that \begin{align*}&[(1,n_1),a_ia_{i+n_1}a_{i+2n_1}\cdots a_{i+I-n_1}]K_0K\\
			=&a_{i+n_1}a_{i+2n_1}\cdots a_{i+I-n_1}\:a_{i+I}\:a_{i+I-n_1}\inv a_{i+I-2n_1}\inv\cdots a_{i}\inv K_0K\\
			=&[a_{i+n_1}a_{i+2n_1}\cdots a_{i+I-n_1},a_{i+I}]K_0K\\
			=&[a_{i+n_1}a_{i+2n_1}\cdots a_{i+I-n_1}a_{i+I},a_{i+I}]K_0K\\
		\end{align*}
		the last of which we had just shown to be central.
		Now to show that $g'$ and $g_1$ commute, observe that $[g',g_1]K_0K=[g',p(g')^i]K_0K=[g',p]K_0K$ where $p\in P_{n'}$. This commutator vanishes for the same reason as outlined above.
		We have thus shown that if $g$ is as above, then $gK_0K$ commutes with $g_1K_0K$ and thus is $gg_1g\inv\not\in g_2K_0K$, concluding the case where $n_1\neq 0$.
		
		Now for the case where $n_1=0$.
		Let again $g$ be such that $gg_1g\inv K_0K=g_2K_0K$, we may write $g=(h,n')$. We first show that $n'\cong 0\mod I$.
		Indeed, suppose the above equality holds, then surely $gg_1g\inv CK=g_2CK=g_1CK$ and thus by \Cref{prop:describeCentralizer}, we have that $h_1\in P_{n'}$.
		However, suppose that $n'\ncong 0\mod I$, then, as $I>2(2\delta+1)$, we may find some integer $j\in [-\delta,\delta]$ such that $j+I\notin [-\delta,\delta]\mod I$. In particular is $h_1\notin P_{n'}$, leading to the conclusion that $n'\cong 0\mod I$. Notice that $(0,I)$ is central in $\frac{G}{K}$, we might thus assume that $n'=0$. In particular we must have that $[h,h_1]K_0K=cK_0K$. As $[h,h_1],c$ and $K_0$ all lie in $C$, and as $C\cap K=1$, this happens precisely when $[h,h_1]K_0=cK_0$ which does not happen by construction of $K_0$.
		We thus have also in the case where $n_1=0$ that $g_1K_0K\nsim g_2K_0K$. 
		
		The group $\frac{G}{K_0K}$ is virtually nilpotent with $N\lelangle t^I\rerangle$ as finite index nilpotent subgroup. By \cite{formanek1976conjugate}, we know thus that $\frac{G}{K_0K}$ is conjugacy separable and thus that there exists a finite quotient $Q$ of $\frac{G}{K_0K}$ in which the images of $g_1$ and $g_2$ are non-conjugate.		
	\end{proof}
	\begin{remark}
	Notice that any finite quotient of $G_0$ factors over a quotient satisfying property $(P)$. Indeed let $\frac{G_0}{K}$ be finite, then $t^P\in K$ for some positive integer $P$. It follows that $$a_iK=t^{P}a_it^{-P}K=a_{i+P}K$$ and thus that $[a_j,a_i]K=[a_j,a_{i+P}]K$. The quotient $\frac{G_0}{K}$ must thus factor over $$\frac{G_0}{\lelangle\{[a_j,a_i][a_j,a_{i+I}]\inv\mid i\in \ZZ\}\rerangle}.$$
	
	Furthermore, notice that the set of groups satisfying property $(P)$ is closed under taking quotients with kernel in $C$. It follows that if $G'\in\cat H$, then also all finite quotients of $G'$ factor over some quotient with property $(P)$.
	\end{remark}
	
	\section{Cyclic centre}\label{sec:conjugacy}
	For our first results, it suffices to consider groups in $\cat H$ with cyclic centre.
	For this we will take a quotient in which the elements $c_i$ get mapped to powers of $c_1$.
	\begin{definition}\label{def:cyclicCentre}
		Let $d:\ZZ\rightarrow \ZZ$ be an anti-symmetric function such that $d(1)=1$. Define $G_d$ the group with presentation
		$$
		G_d=\lelangle \{t,a_i,c_i\vert i\in\ZZ\}\Bigg\vert \begin{matrix} ta_it\inv=a_{i+1}&i\in\ZZ\\
			tc_it\inv=c_i&i\in\ZZ\\
			[a_i,a_j]=c_{i-j}&i,j\in\ZZ\\ [c_i,a_j]=[c_i,c_j]=1&i,j\in\ZZ\\ c_i=c_1^{d(i)}&i\in\ZZ
		\end{matrix}\rerangle.
		$$
	\end{definition}
	Notice that $G_d$ is indeed a quotient of $G_0$ where the kernel is generated by the central elements $\{c_ic_1^{-d(i)}\mid i\in\ZZ\}$, and as such $G_d\in\cat H$.
	
	The properties $G_d$ attains depend on the function $d$. One first example of this is the following:

	\begin{proposition}\label{prop:whenRf}
		The group $G_d$ is residually finite if and only if $d\mod q$ is periodic for infinitely many positive integers $q$.
	\end{proposition}
	\begin{proof}
		Let $g\in G_d$ be non-trivial.
		As $\ZZ\wr\ZZ$ is residually finite, we may restrict ourselves to the case where $g=c_1^k$. Let $q>\abs{k}$ be such that $d\mod q$ is $I$-periodic, then $g$ is non-trivial in $\frac{G_d}{\lelangle c^q\rerangle}$. The result follows as by \Cref{prop:cSepPeriodicCentre}, this quotient is conjugacy separable and thus residually finite.
		
		On the other hand, let $k$ be a positive integer, and let $\varphi:G_d\rightarrow Q$ be a finite quotient of $G_d$ in which $c_1^k$ does not vanish. Let $q$ be the order of $\varphi(c_1)$ in Q. Notice that $q\nmid k$. Let $I_0$ be minimal such that $\varphi(a_i)=\varphi(a_{i+I})$ for some $i$, after conjugation with $t$, we have that $\varphi(a_ia_{i+I}\inv)$ is trivial for all $i$. Taking the commutator with $\varphi(a_0)$, we obtain that $\varphi[a_ia_{i+I}\inv,a_0]=\varphi(c_1^{d(i)-d(i-I_0)})$ is trivial, and thus that $q\mid d(i)-d(i-I_0)$ for every $i$. In particular is $d\mod q$ $I$-periodic.
		
		Let $k$ be the product of all $q$ for which $d\mod q$ is periodic, then there can be no finite quotient in which $c_1^k$ does not vanish.
	\end{proof}
	A recent preprint \cite{bodart2025finite} considered the class of central $\frac{\ZZ}{2\ZZ}$-by-$\frac{\ZZ}{2\ZZ}\wr H$ extensions. Note that if $H=\ZZ$, then these groups are quotients of $G_d$ for some $d$ with kernel generated by $\{a_i^2,c_1^2\mid i\in\ZZ\}$. Their Proposition 5.1. Then translates to $d\mod 2$ being periodic. Which of course closely resembles our \Cref{prop:whenRf}.
	
	We also formulate a variation of \Cref{prop:whenRf}
	
	\begin{proposition}\label{prop:conjsep}
		$G_d$ is conjugacy separable if and only if $d\mod q$ is periodic for every positive integer $q$.
	\end{proposition}

	\begin{proof}
		Assume $d\mod q$ is periodic for every $q$.
		Let $g_1,g_2\in G_d$ be non-conjugate. As $\ZZ\wr\ZZ$ is conjugacy separable, we may assume that $g_2=g_1c_1^{i}$ for some $i\neq 0$. Let $J=\{j\in \ZZ\mid g_1c_1^{j}\sim g_1\}$. As $c_1$ is central in $G_d$, we have that $J$ is a subgroup of $\ZZ=\langle c_1\rangle$. Using that $\ZZ$ is subgroup separable, combined with the Chinese remainder theorem, we may find some prime power $q=p^k$ such that $i\notin J\mod q$. It follows that $g_2 \langle c_1^q\rangle\nsim g_1\langle c_1^q\rangle$.
		The result follows from \Cref{prop:cSepPeriodicCentre} as $\frac{G_d}{\langle c_1^q\rangle}$ satisfies property $(P)$.
		
		For the other direction, notice first that if $d\mod q$ is periodic for every prime power $q$, then $d\mod m$ must also be periodic for every positive integer $q$. Assume thus that $q=p^k$ is such that $d\mod q$ is not periodic. Consider the elements $g_1=a_0^{p^k}$ and $g_2=a_0^{p^k}c_1^{p^{k-1}}$. One sees easily that $a_0^{p^k}$ is conjugate to $a_0^{p^k}c_1^{I}$ precisely when $p^k\mid I$.
		We will show that $g_1$ and $g_2$ are conjugate in every finite quotient of $G_d$.
		
		Pick $Q$ a finite quotient of $G_d$ and let $I_0$ and $I_1$ be such that $a_{I_0}$ and $a_{I_0+I_1}$ get identified. After conjugation with $t^{I_0}$ we obtain that $a_0$ and $a_{I_1}$ get identified in $Q$. As $d\mod q$ is not $I_1$ periodic, there exists some $j$ such that $d(j)\neq d(j-I_0)\mod p^k$. In particular, there exists some $u\in \ZZ$ such that $u(d(j)-d(j-I_1))=p^{k-1}\mod p^k$.
		
		As $a_0\inv a_{I_1}$ gets mapped to the identity in $Q$, so must $[a_0\inv a_{I_1},a_j^{u}]$ be mapped to the identity. This commutator is given by $c_1^{u(d(j)-d(j-I_1))}$. We thus also obtain that $g_1$ gets identified with $a_1^{p^k}c_1^{p^{k-1}+u(d(j)-d(j-I_1))}$. This element is conjugate to $g_1$.
		
		\end{proof}

	As periodicity of $d$ is so important for our applications, we hard code it into our construction.
	\begin{definition}\label{def:d}
		Let $f:\NN\rightarrow\NN$ be a sequence such that $f(0)=1$, define then $d_f:\ZZ\rightarrow\ZZ$ as $$
		d_f(i)=\begin{cases}
			\begin{matrix}
				0&i=0\\
				f(j)&i\cong3^j\mod 3^{j+1}\\
				-f(j)&i\cong -3^j\mod 3^{j+1}.
			\end{matrix}
		\end{cases}
		$$
	\end{definition}
	\begin{example}
		If $f:\NN\rightarrow\NN$ where to map $i$ to $i+1$, then $d_f$ would be the odd function given on $\NN$ by
		$$
		0,1,-1,2,1,-1,-2,1,-1,3,1,-1,2,1,-1,-2,1,-1,-3,1,-1,2,1,-1,-2,1,-1,4\cdots.
		$$
	\end{example}
	Clearly, the above procedure is algorithmic:
	\begin{lemma}\label{prop:recursivelyPresented}
		Let $f:\NN\rightarrow\NN$ be a computable function, then $G_{d_f}$ is recursively presented
	\end{lemma}
	\begin{proof}
		First notice that if $f$ is computable, then so is $d$.
		Using the relations of the first type of $G_{d_f}$, we have that $a_i=t^{i}a_0t^{-i}$. Substituting this in the relations gives us $5$ families of relations, the first $4$ of which are obviously computable sets, and the last being computable as $d$ is.
	\end{proof}
	A further advantage of our chosen form of $d$, is that it is often easier to describe commutators in $N_d$, this is exemplified by the following two lemmas:
	
	In the following lemma $\nu_3$ denotes the $3$-adic valuation, that is for $n$ an integer, $\nu_3(n)$ is maximal such that $3^{\nu_3(n)}\mid n$.
	\begin{lemma}\label{prop:comutWithSeries}
		Let $n_0>0$ and for $i\in\NN$, let $$g=a_0a_{n_0}a_{2n_0}\cdots a_{(3^{i}-1)n_0}\in G_{d_f}.$$
		Then for $j\in\ZZ$ such that $\nu_3(n_0)\leq \nu_3(j)$, $[g,a_j]=c^{d_f(j-\tilde j)}$ where $\tilde j$ is the unique integer in $\{0,n_0,2n_0,\cdots,(3^i-1)n_0\}$ that is congruent to $j \mod 3^{i+\nu_3(n_0)}$.
	\end{lemma}
	\begin{proof}
		Notice that $[g,a_j]$ is given by $$c^{\sum_{n=1}^{3^i-1} d(j-nn_0)}.$$
		Notice that for any $n\in[0,3^{i}-1]$ there exists a unique $n'$ in that same interval such that ${nn_0-j\cong-(n'n_0-j)\mod 3^{i+\nu_3(n_0)}}$. In particular, if ${nn_0-j\not\cong 0\mod 3^{i+\nu_3(n_0)}}$, we have that $d(j-nn_0)=-d(j-n'n_0)$.
		Cancelling out these terms we are left with $$
		[g,a_j]=c^{d(j-\tilde j)}.
		$$
		as had to be shown
	\end{proof}
	
	\begin{lemma}\label{prop:SelectorCommutator}
		Let $g_1=\prod_{i=-m}^m a_i^{e_i}$ Let $j_0\in [-m,m]$ and let $j\in\NN$ be such that $3^j>2m$. Then $[g_1,a_{j_0+3^j}\inv a_{j_0}]=c_1^{e_{j_0}f(j)}$.
	\end{lemma}
	\begin{proof}
		By definition, $[g_1,a_{j_0}\inv a_{j_0+3^j}]$ is given by $c_1^{\sum e_id(i-j_0)+d(j_0+3^j-i)}$. When $i\in [-m,m]\backslash \{j_0\}$, then $j_0-i\cong j_0+3^j-i\ncong 0\mod 3^j$. In particular, $d(j_0-i)=d(j_0+3^j-i)$. We are thus left with $c_1^{e_i(d(i-i)+d(i+3^j-i)}$ which is in turn equal to $c_1^{e_{j_0}d(3^j)}$.
	\end{proof}
	\subsection{Separating word and conjugacy problem}\label{sec:function}
	In this subsection, we construct a group that is conjugacy separable with solvable word problem but unsolvable conjugacy problem.
	We start with defining the function $f$ that we wish to plug into $G_{d_f}$.
		
	To specify the function $f$ we will use, first let $\mc P$ be a prime valued function such that $P$ is computable, but such that membership of the image $\mathrm{Im}(P)$ is undecidable.
	\begin{lemma}
		There exists a computable, prime valued function $\mc P:\NN\rightarrow\NN$ such that membership of $\mathrm{Im}\mc P$ is undecidable.
	\end{lemma}
	\begin{proof}
		For instance, one can index all Turing machines with a prime number, then $\mc P$ lists the primes corresponding to all Turing machines that halt in the order where the number of steps taken plus the index of the machine increases.
	\end{proof}

	Let $f_0:\NN\rightarrow\NN$ be the function that maps $i$ to $\prod_{j<i}p_j$ where $p_*$ is the sequence of primes.
	Let $\mc P':\NN\rightarrow\NN$ be defined case-wise as$$
	\mc P':i\mapsto\begin{cases}
		\begin{matrix}
			1&\mc P(i)\nmid f_0(i)\\
			1&\exists j<i:\mc P(j)=\mc P(i)\\
			\mc P(i)&\text{otherwise}
		\end{matrix}
	\end{cases}
	$$
	and define $f:\NN\rightarrow\NN$ such that $f(i)=\left(\frac{f_0(i)}{\mc P'(i)}\right)^i$.
	\begin{lemma}\label{prop:dComputable}
		$f$ is a computable function.
	\end{lemma}
	\begin{proof}
		This follows from $\mc P$ being a computable function, and all further manipulations being computable.
	\end{proof}
	\begin{corollary}
		$G_{d_f}$ is recursively presented finitely generated
	\end{corollary}
	\begin{proof}
		This follows from the above and \Cref{prop:recursivelyPresented}.
	\end{proof}
		
	\begin{lemma}
		Let $q=p^k$ be a prime power, then $d_f\mod q$ is a periodic function.
	\end{lemma}
	\begin{proof}
		The function $\mc P'$ obtains $p$ as value at most once. It follows that $f$ takes multiples of $q$ as values except in a finite number of cases. Let $I$ be such that for $i\geq I$, $q\mid f(i)$.
		If $j\not\cong 0\mod 3^{I}$, then $d_f(j)$ is completely determined by this congruence class. Furthermore, if $j\cong 0\mod  3^{I}$, then $d_f(j)$ must take as value a multiple of $q$. It follows that $d_f\mod q$ is $3^{I}$ periodic.
	\end{proof}
	It follows by \Cref{prop:conjsep} that $G_{d_f}$ is conjugacy separable.
	\begin{remark}
		Even though this period always exists, it might be impossible to determine it effectively as for some primes $p$, it is impossible to determine whether or not $n_0$ takes them on as values.
	\end{remark}
	
	Now we proceed by demonstrating that the word problem is solvable in $G_{d_f}$
	\begin{proposition}
		The word problem is solvable in $G_{d_f}$.
	\end{proposition}
	\begin{proof}
		Let $w\in F_{\{\tau,\alpha\}}$. Notice that $w(t,a_0)\in N$ if and only if the sum of all the exponents of $\tau$ in $w$ is $0$. Applying relations of the form $ta_it\inv=a_{i+1}$, we may find some word $w'\in F_{\{\alpha_i\mid i\in \ZZ\}}$ such that $w'((a_i)_i)=w(a_0,t)$. Let $A\subset \ZZ$ be finite such that $w'$ contains letters only in $\{a_i,a_i\inv\mid i\in A\}$. Then obviously $w'((a_i)_i)=w(a_0,t)$ vanishes in $G_{d_f}$ if and only if it vanishes in ${N'=\langle\{a_i,a_i\inv\mid i\in A\}\rangle}$. We can find an explicit presentation for $N'$ by computing $d(j-i)$ for all $i,j\in A$. Finally the result follows as the word problem is solvable for finitely presented $2$-step nilpotent groups.
	\end{proof}

	\begin{lemma}\label{prop:TranslateConjToMembership}
		Let $i,q=p^k$ be such that $q\mid f_0(i)$. Let $$
		g_1=a_0a_1a_2\cdots a_{3^{i}-1}\in G_{d_f}
		$$
		and let $g_2=g_1c_1^{\frac{f(i)}{p^{i}}}$. Then are $g_1$ and $g_2$ conjugate if and only if $\mc P'(j)=p$ for some $j>i$.
	\end{lemma}
	\begin{proof}
		
		By \Cref{prop:conjInSubgroup}, $g_1\sim g_2$ if and only if $c_1^{\frac{f(i)}{p^{i}}}\in [g_1,N]$. Furthermore, by \Cref{prop:comutWithSeries}, $[g_1,N]$ is given by $\langle c_1^{f(j)}\mid j\geq i\rangle$. In particular this subgroup contains $c_1^{f(i)}$. By the Chinese remainder theorem, it follows that it contains  $c_1^{\frac{f(i)}{p^{i}}}$ if and only if it contains some $c_1^k$ where $k$ is not a multiple of $p$. This happens precisely if there exists some $j>i$ such that $p\nmid f(j)$ which happens precisely when $\mc P'(j)=p$.
	\end{proof}
	\begin{corollary}
		The conjugacy problem is not solvable in $G_{d_f}$.
	\end{corollary}
	\begin{proof}
		Assume from contradiction that the conjugacy problem is solvable in $G_{d_f}$, we will show that the membership problem in $\mathrm{Im}(\mc P)$ is decidable for primes.
		Let $p$ be the $i$-th prime. As $\mc P$ is a computable function, we may compute $\mc P(j)$ for $j\leq i$ explicitly. If $p=\mc P(j)$ for any $j\leq i$, then we are done. We may thus assume that  $p\neq\mc P(j)$ for any $j\leq i$.
		
		As the conjugacy problem is solvable, we may determine if $$
		g_1=a_0a_1a_2\cdots a_{3^{i}-1}
		$$ and $$g_2=g_1c_1^{\frac{f(i)}{p^{i}}}$$
		are conjugate in $G_{d_f}$. By \Cref{prop:TranslateConjToMembership}, this allows us to determine if $\mc P'(j)=p$ for some $j>i$. But, as $p\neq\mc P(j)$ for any $j\leq i$,  this happens precisely when $p\in \mathrm{Im}(\mc P)$. It would thus follow that the membership problem for $\mathrm{Im}(\mc P)$ is decidable for primes. This is not the case by definition of $\mc P$.
		The result follows from contradiction.
	\end{proof}
	Combining the results of this section we obtain 
	\setcounter{result}{0}
	\begin{result}\label{prop:ExistsConjNotSolvable}
		There exists a finitely generated recursively presented conjugacy separable $3$-step solvable group with solvable word problem, but unsolvable conjugacy problem.
	\end{result}

	\subsection{Conjugator length and conjugacy separability}
	As another application of our construction, in this subsection we construct a group with large conjugacy separability growth, but small residual finiteness growth, efficient conjugacy problem and small conjugator length. These concepts are all related as becomes apparent when studying the time complexity of McKinseys algorithm for the conjugacy problem. Morally the statement gives that the length of conjugators in finite quotients are longer then what would be expected from the length of conjugators in the original group.
	
	We again start by defining a sequence $f:\NN\rightarrow\NN$ as in \cref{sec:function}.
	Let $p_i$ be a computable enumeration of the primes such that $p_2=3$ (for instance the usual enumeration)\begin{itemize}
	\item The sequence $f_1$ is valued in the powers of $p_1$ and is weakly increasing, but slowly. The growth rate of this function will control the conjugacy separability growth. Morally, $f_1$ should increase as slowly as possible, but for now we only lay the constraint that $f_1(i)\leq \exp(i)$.
	Furthermore, $f_1$ should be computable in polynomial time.
	
	\item The sequence $f_2$ is exponential and valued within the powers of $p_2$, that is $f_2(i)=p_2^{i}$. 
	This second sequence will bound the residual finiteness growth.
	
	\item The sequence $f_3$ is weakly increasing such that $p_1,p_2$ do not divide any entry of $f_3$, such that $f_3(i)\leq \exp(i)$, such that $f_3(i)\mid f_3(i+1)$, and such that for any integer $q$ with $p_1,p_2\nmid q$, there exists some integer $i_q$ such that $q\mid f_3(i_q)$.
	This final sequence will guarantee that the group we end up with is conjugacy separable.
	\end{itemize}
	\begin{lemma}
		Such a sequence $f_3$ exists and can be chosen to be computable.
	\end{lemma}
	\begin{proof}
		We construct the sequence $f_3$ inductively, for this we construct a sequence of triples $(m,q,l)_i$ where $m_i$ is an initial segment of a sequence of length $l_i$ such that $m_i(l_i)=q_i$. We start with 
		$m_0=(1), q_0=1$ and $l_0=1$. Now we proceed inductively: 
		Let $l_{i}=2l_{i-1}+\lceil\log(p_{i+2})\rceil$, let $q_i=q_{i-1}^2p_{i+2}$ and let $m_i$ be the finite sequence obtained by concatenating $m_{i-1}$, a sequence of $l_{i-1}+\lceil\log(p_{i+2})\rceil-1$ entries of $q_{i-1}$'s and a single $q_i$.
		Notice for $i<j$, $m_i$ is an initial segment of $m_j$. We can thus define $f_3$ to be the limit sequence of $m_i$, or thus the sequence such that for $i,j\geq 0$ such that $l_j\geq i$, $f_3(i)=m_j(i)$.
		We show that $f_3$ has the desired properties.
		
		Weakly increasing follows as each of the sequences $m_i$ is weakly increasing. The primes $p_1$ and $p_2$ do not divide $q_0$ and by induction they divide none of the $q_i$, if follows that if $m_i$ contains no multiples of $p_1,p_2$.
		
		Clearly $q_0=1\leq \exp(l_0)$ and by induction, if $q_{i-1}\leq \exp(l_{i-1})$, then $$q_i^2p_{i+2}\leq \exp(2l_{i-1}+\lceil \log(p_{i+2})\rceil)=\exp(l_i).$$
		
		Finally, let $p_i^{2^j}$ be some prime power, then $p_i$ divides $q_{i-2}$, and thus is $p_i^{2^j}\mid q_{i+j-2}$.
	\end{proof}
	We now define $f(i)=f_1(i)f_2(i)f_3(i)$, and we define $d=d_f:\ZZ\rightarrow\ZZ$ as in \cref{def:d} by 
	$$
	d(i)=\begin{cases}
		\begin{matrix}
			0&i=0\\
			f(j)&i\cong3^j\mod 3^{j+1}\\
			-f(j)&i\cong -3^j\mod 3^{j+1}.
		\end{matrix}
	\end{cases}
	$$
	Notice that if $f_1,f_2$ and $f_3$ are computable, then so is $d$. Furthermore, by construction of $f$, we have by \Cref{prop:conjsep} that $G_d$ is conjugacy separable.
	
	We begin by demonstrating that the conjugacy separability growth is large.
	Denote $f_1\inv$ the right semi-inverse of $f_1$, that is the function $f_1\inv:\NN\rightarrow\NN$ such that 
	$$f_1\inv(i)=\min\{j\in\NN\mid f_1(j)\geq i\}.$$
	\begin{proposition}
		Let $d$ be as above then $\conj_{G_d}\succ 3^{f_1\inv}$.
	\end{proposition}
	\begin{proof}
		Let $q$ be some power of $p_1$, consider then $g_1=a_0^q$ and $g_2=a_0^qc_1^{\frac{q}{p_1}}$.
		Notice that $g_1$ and $g_2$ are non-conjugate. 
		Conjugating with powers of $a_1$ shows that $g_1$ is conjugate to every element of the form $a_0^qc_1^{kq}$.
		Let $\varphi:G_d\rightarrow Q$ be a finite quotient, and suppose the order of $Q$ is less than $3^{f_1\inv(q)}$. In particular, the order of $\varphi(t)$ is less than $3^{f_1\inv(q)}$, call this order $I$. As $\varphi(t^I)=1$, it follows that $\varphi([t^Ia_0t^{-I},a_0])=1$ and thus that $\varphi(c_1^{d(I)})=1$. As $\abs{I}\leq 3^{f_1\inv(q)}$, it follows that $d(I)$ is not a multiple of $q$. In particular is $\frac{q}{p_1}=kq\mod d(I)$ for some $k\in\ZZ$ or thus is $\varphi(g_1)$ conjugate to $\varphi(g_2)$.
		As $g_1$ and $g_2$ have word norm at most $q+6$ with respect to $S$, it follows that $\conj_{G_d,S}(p_1i+6)\geq 3^{f_1\inv(i)}$.
	\end{proof}
	Contrary to the conjugacy separability growth, the residual finiteness growth remains tame. To demonstrate this, we first show that $\abs{d(i)}$ is not too large compared to $i$.
	\begin{lemma}\label{prop:dBound}
		There exists some integer $\delta$ such that $\abs{d(i)}^\delta\leq \abs{i}$.
	\end{lemma} 
	\begin{proof}
		First notice that $f(i)\leq \exp(ci)$ for some constant $c$, this follows immediately from the conditions on $f_1,f_2$ and $f_3$.
		Notice furthermore that for $i$ with $\abs{i}<3^{j_0}$, $\abs{d(i)}=f(j)$ for some $j<j_0$. It thus follows that $\abs{d(i)}\leq f(\lfloor\log_3{i}\rfloor)\leq \exp(c\lfloor\log_3{i}\rfloor).$
		The result follows.
	\end{proof}
	\begin{remark}
		The constant $\delta$ can be lowered by tweaking the definitions of $f_1,f_3$.  If one maintains $p_2=3$, then any value of $\delta>1$ can be obtained for some sequence $f_1,f_2$.
	\end{remark}
	\begin{corollary}\label{prop:centerSizeBound}
		Let $c_1^k\in G_d$ be of norm $\norm{c_1^k}_S\leq m$, then for $\delta$ as above we have $\abs{k}\leq cm^{2+\delta}$. Where $c$ is some constant independent of $k$.
	\end{corollary}
	\begin{proof}
		 This is immediate as by $\Cref{prop:supportC}$, $c_1^k$ is of norm at most $m^2$ with respect to the generating set
		 $\{c_1^{\pm f(i)}\mid 3^i\leq m\}$.		 
		 
	\end{proof}
	\begin{proposition}
		Let $\delta$ be as above, then $\rf_G\prec \left(n^{2+\delta}\right)^{n^{2+\delta}}$.
	\end{proposition}
	\begin{proof}
		Let $g\in G_d$ be non-trivial and suppose $\norm{g}_S\leq m$.
		By \Cref{prop:rflamplighter}, we know that $\rf_{\ZZ\wr\ZZ}$ is at most polynomial, so we might restrict ourselves to the case where $g=c_1^k$. Furthermore by \Cref{prop:centerSizeBound} we know that $\abs{k}\leq cm^{2+\delta}$. Let $q$ be some power of $p_2$ such that $k<q\leq kp_2$. Let $I$ be a period of $d\mod q$, by construction of $f_2$, such a period is given by $q$ itself (where we use that $p_2=3$).
		
		Let $K$ be the normal subgroup of $G_d$ given by $\llangle \{t^{q}, a_i^q,a_ia_{i+q}\inv,c_1^q\mid i\in \ZZ\}\inv \rrangle$. This subgroup intersects $C$ in $C^q$ and thus is $g=p_1^k$ not an element of $K$.
		The order of $\frac{G_d}{K}$ is given by $q^{2+q}$ which is bounded by $$
		\left(3cm^{2+\delta}\right)^{2+3Cm^{2+\delta}}.
		$$
		The result follows.
	\end{proof}
	
	Now what is left is to show that the conjugacy problem on $G_d$ is efficient, and that the conjugator length function is small.
	For both of these, given two elements $g_1,g_2\in G_d$, we first reduce to the case where $g_1C$ and $g_2C$.
	\begin{lemma}\label{prop:conjugatorZwrZ}
		The conjugacy problem of $\ZZ\wr\ZZ$ is polynomial time solvable, furthermore if $g_1,g_2\in\ZZ\wr\ZZ$ are conjugate, then one finds in polynomial time an element $h\in\ZZ\wr\ZZ$ of norm at most linear in $\norm{g_1},\norm{g_2}$ such that $hg_1h\inv=g_2$.
	\end{lemma}
	\begin{proof}
		Polynomial time solvability of the conjugacy problem follows from \cite{vassileva2011polynomial}, the length of $h$ follows from \cite{sale2016conjugacy}. They also give an explicit construction of $h$ which is clearly computable in polynomial time.
	\end{proof}

	\begin{lemma}\label{prop:narrowConjugator}
		Let $f':\NN\rightarrow\NN$ be such that $f'(i)\mid f'(i+1)$
		Let $g_1,g_2\in G_{d_{f'}}$, such that $g_1C=g_2C$ with $\norm{g_1}_S,\norm{g_2}_S\leq m$. Suppose that $g_1,g_2\in N$ and suppose that $g_1\sim g_2$, then there exists some $h\in \langle\{a_i\mid \abs{i}\leq \tilde m\}\rangle$ such that $hg_1h\inv=g_2$ where $\tilde m$ is the least power of $3$ such that $\tilde m>2m+1$.
	\end{lemma}
	\begin{proof}
		By \Cref{prop:conjInSubgroup}, $g_1\sim g_2$ if and only if $g_2g_1\inv\in [g_1,N]$.
		A set of generators of $N$ can be given by $$\{a_i\mid \abs{i}\leq \tilde m \}\cup \{a_{i'}a_i\inv\mid \abs{i}\leq \tilde m, \abs{i'}\geq m, i\cong i'\mod \tilde m\}$$ and thus is a set of generators of $[g_1,N]$ given by $$\{[g_1,a_i]\mid \abs{i}\leq \tilde m \}\cup \{[g_1,a_{i'}a_i]\inv\mid\abs{i}\leq \tilde m, \abs{i'}\geq m, i\cong i'\mod \tilde m\}.$$
		
		If we write $g_1$ as $a_{-m}^{e_{-m}}a_{-m+1}^{e_{-m+1}}\cdots a_{m}^{e_{m}}c^k$, then by \Cref{prop:SelectorCommutator}, the second of these sets is given by$$
		\{
		c_1^{e_if'(\log_3(i'-i))}\mid \abs{i}\leq m, i\cong i'\mod \tilde m 
		\}
		$$
		As $f'(j)\mid f'(j+1)$, we have in particular that the above subgroup is also generated by the smaller set$$
		\{
		c_1^{e_if'(\log_3(\tilde m))}\mid \abs{i}\leq m, \abs{j}\geq m
		\}
		$$
		which is also given by $\{[g_1,a_{i'}a_i\inv]\mid i\in [-m,m], i'\in[-\tilde m,\tilde m]\}$, where we use that for any $i\in [-m,m]$, there exists some $i'\in [-\tilde m,\tilde m]\backslash[-m,m]$ such that $\abs{i-i'}=\tilde m$.  It thus follows that $[g_1,N]$ is contained in (equal to) $[g_1,\langle{S_{\tilde m}}\rangle]$.
		In particular, if there exists some $h$ such that $g_2g_1\inv=[g_1,h]$, then it can be chosen in $\langle{S_{\tilde m}}\rangle$.

	\end{proof}
	From the above lemma, it follows quite straightforwardly that both the conjugacy problem is efficient, and that conjugators are short.
	For the latter, we use an effective version of Bézout's identity.
	\begin{lemma}\label{prop:gcdintegers}\change{lemma got adapted}
		Let $e_0,e_1,\cdots e_m\in \ZZ$, with greatest common divisor $e$. Let $e'$ be some multiple of $e$ such that $\abs{e'}\leq\min(\abs{e_i})$, then there exists $k_0,k_1,\cdots k_m$ in $\ZZ$ such that $\sum\abs{k_i}\leq \sum\abs{e_i}$ and such that $\sum{e_ik_i}=e'$.
	\end{lemma}
	\begin{remark}\change{This is a new remark}
		In particular the above holds for $e'=e$.
	\end{remark}
	\begin{proof}\change{proof got adapted}
		Without loss of generality, we may assume that the $e_i$ are distinct and in ascending order. By Bézout's identity, there exist coefficients $k_i$ such that $\sum{e_ik_i}=e'$. If we add $e_i$ to $k_0$ and subtract $e_0$ from $k_i$, then we again maintain $\sum{e_ik_i}=e'$.
		Pick $\tilde k_i$ for $i\geq 1$ such that $\tilde k_i\cong k_i\mod e_0$, $\abs{\tilde k_i}\leq\abs{e_0}$ and such that $\sum_{i=1}^je_i\tilde k_i-e'\leq\abs{e_0e_j}$ for all $j$. One shows inductively that such $\tilde k_i$ exist. As $\tilde k_i\cong k_i\mod e_0$ for all $i>0$, we have that $\sum_{i=1}^m\tilde k_ie_i-e'$ is a multiple of $e_0$. Furthermore, by construction of $\tilde k_i$, we have have that $\frac{\sum_{i=1}^m\tilde k_ie_i-e'}{e_0}$ is bounded in absolute value by $e_m$.
		The result follows.
	\end{proof}
	This gives us the following:
	\begin{proposition}\label{prop:ConjugatorShort}
		Let $d$ be as before, then $G_d$ has polynomial conjugator length.
	\end{proposition}
	\begin{proof}
		Let $g_1,g_2\in G_d$ be conjugate and both of length at most $m$. By \Cref{prop:conjugatorZwrZ}, we may assume that $g_1C=g_2C$, replacing $g_2$ by some conjugate and $m$ by some multiple of $m$ if necessary.
		
		Suppose that $g_1\notin N$, then it follows by \Cref{prop:conjInSubgroup} that $g_1\sim g_2$ if and only if $g_1=g_2$. In this case we are thus done.
		
		We now assume that $g_1,g_2\in \langle S_m\rangle$. By \Cref{prop:narrowConjugator}, it follows that a conjugator can be found in $N_{\tilde m}=\langle S_{\tilde m}\rangle$. It thus also follows that $g_2g_1\inv\in [g_1,N_{\tilde m}]$. By \Cref{prop:centerSizeBound}, we have that $g_2g_1\inv = c_1^k$ where $k\leq cm^{2+\delta}$ for $c$ some uniform constant. If we denote $c_1^{k_i}=[g_1,v_i]$, then we have by the same \namecref{prop:centerSizeBound} that $k_i\leq c \tilde m^{2+\delta}$. 
		By \Cref{prop:gcdintegers}, we can find coefficients $x_i$ such that $\sum x_ik_i=\gcd\{k_i\mid i\in [-\tilde m,\tilde m]\}$ and such that $\sum \abs{x_i}\leq \sum \abs{k_i}$. A conjugator $h$ is then given by $$
		h=\prod_{i\in [-\tilde m,\tilde m]} a_i^{\frac{x_ik}{\gcd\{k_j\mid j\in [-\tilde m,\tilde m]\}}}
		$$
		which is by the above considerations of polynomial length in $m$.
	\end{proof}
	\begin{proposition}
		Let $d$ be as before, then $G_d$ has polynomial time solvable conjugacy problem.
	\end{proposition}
	\begin{proof}
		Let $g_1,g_2\in G_d$ of length $m$. As the conjugacy problem is polynomial time solvable in $\ZZ\wr\ZZ$, we might assume that $g_1\sim g_2$, furthermore, using \Cref{prop:conjugatorZwrZ}, we find in polynomial time some conjugator $h$ of length linear in $m$ such that $hg_1hC\inv=g_2C$. An explicit formulation of this conjugator is given and it can be computed in polynomial time.
		
		We might thus, as in the proof of \Cref{prop:ConjugatorShort}, assume that $g_1C=g_2C$. Again as in that proof we might assume that $g_1,g_2\in \langle S_m\rangle$.
		In this case it remains to check that given $k$ such that $c_1^k=g_2g_1\inv$ and $k_i$ such that $c_1^{k_i}=[g_1,v_i]$ for $i\in \tilde m$ that $\gcd\{k_i\mid i\in [-\tilde m,\tilde m]\}\mid k$. 
		Notice first that $k$ and $k_i$ can be computed exactly in polynomial time, and that their size is at most polynomial in $m$. It follows that $\gcd\{k_i\}\mid k$ can also be checked in polynomial time.
	\end{proof}
	Combining the results from this sections, we obtain the following
	\begin{result}\label{prop:cLengthCojSep}
		There exists a family of finitely generated recursively presented torsion-free 3-step solvable conjugacy separable groups $G_d$ such that:\begin{itemize}
		\item$\rf_{G_d}\prec \exp(n^\delta)$ for some integer $\delta$; 
		\item$\cl_{G_d}$ is at most polynomial; \item the conjugacy problem is polynomial time solvable; \item for any computable function $\Phi$, there exists some $d$ such that $\conj_{G_d}\succ \Phi$.
		\end{itemize}
	\end{result}
	\begin{proof}
		The only thing that is left to show is that for every computable function $\Phi$ that goes to $\infty$, there exists a smaller function $\Phi'$ that also goes to infinity and that is computable in polynomial time. For this take $\Phi'(n)=\Phi(i)$ where $i$ is maximal such that $\Phi(j)$ can be computed for every $j\leq i$ in time at most $n$. 
	\end{proof}
	\subsection{Arbitrary conjugator lengths}\label{sec:conjugatorLength}
	In this section we specify our groups from \ref{def:cyclicCentre} to obtain arbitrary functions as conjugator length functions.
	\begin{definition}\label{def:dArbitraryConjugatorLength}
		Let $\Phi:\NN\rightarrow\NN$ be a non decreasing function such that $\Phi(i)\geq i^3$.
		Let $p_i$ be the sequence of primes in usual order. Define then $f:\NN\rightarrow\NN$ as $$
		f(j)=\prod\{p_i\mid 3^i\leq 3^j\leq f(3^i) \}
		$$	
	\end{definition}
	We then define $d_f:\ZZ\rightarrow\ZZ$ as in \Cref{def:d}, we will consider the groups $G_{d_f}$
	We will show that $\cl_{G_{d_f}}\simeq \Phi$. when $\Phi$ is sufficiently large.
	Our result will hold for $\Phi$ one of the following: 
	\begin{enumerate}[label=(\alph*)]
		\item $\forall i\in \NN, \Phi(i)\geq i^3$ and $\Phi(3i)\geq 3\Phi(i)$;\label{cond:LocalLarge}\\
		\item $\Phi(i)\geq (i)^{1+\log_3(2)+\log_3\log_3(i)}$.\label{cond:GlobalLarge}\\
	\end{enumerate}
	Through most of our strategy, there is no need to distinguish between the $2$ cases.
	Let $\tilde \Phi$ be the unique function such that $\tilde \Phi$ is always a power of $3$ and such that $\Phi(i)\leq \tilde \Phi(i)< 3\Phi(i)$ for all $i$. Clearly $\Phi\simeq \tilde \Phi$. Thus it suffices to demonstrate $\cl_{G_{d_f}}\simeq \tilde \Phi$. Notice furthermore that $d$ does not change if we replace $\Phi$ by $\tilde \Phi$.
	For notational convenience, define furthermore $\Phi_0$ such that $3^{\Phi_0(i)}=\tilde\Phi(3^i)$.
	
	We have some bounds on $d$.
	\begin{lemma}\label{prop:nUpperBound}
		Let $i\in \ZZ$ and let $j\in\ZZ$ be maximal such that $3^j\leq \abs{i}$, then $\abs{d(i)}\leq p_j^j$.
	\end{lemma}
	\begin{proof}
		This is immediate as $f(j)\leq \prod_{i=1}^jp_i$.
	\end{proof}
	\begin{lemma}\label{prop:dLowerBound}
		Let $j\in\NN_{>9}$, then $d(\tilde \Phi(3^j))\geq p_j^j3^{2j}$.
	\end{lemma}
	\begin{proof}
		By construction, $d(\tilde \Phi(3^j))$ is at least $\prod\{p_i\mid j\leq i\leq \Phi_0(j)\}$. Notice that $p_i\geq p_j$. Notice furthermore that $\tilde \Phi(3^j)\geq 3^{3j} $. In particular, we have that $\{p_i\mid j\leq i\leq \Phi_0(j)\}$ contains at least $2j$ elements.
		It follows that $d(\tilde \Phi(3^j))\geq p_j^{2j}\geq p_j^{j}3^{2j}$.
	\end{proof}
	\begin{lemma}\label{prop:gcd1}
		$\gcd\{f(j),f(1+\Phi_0(j)))\}=1$
	\end{lemma}
	\begin{proof}
		All prime factors of $f(j)$ are at most $p(j)$ and the prime factors of $f(1+\Phi_0(j))$ are at least $p_{j+1}$.
	\end{proof}
	We begin by demonstrating a lower bound to the conjugator length function.
	\begin{lemma}\label{prop:conjWidthLowerBound}
		Let $g_1=a_0a_1a_2\cdot a_{3^j-1}$ and let $g_2=g_1c_1$. Let $N_j=\langle S_{\tilde \Phi(3^j)}\rangle$ and $N_j^+=\langle S_{3\tilde \Phi(3^{j})}\rangle$, then $g_2\notin g_1^{N_j}$ but $g_2\in g_1^{N_j^+}$. 
	\end{lemma}
	\begin{proof}
		By \Cref{prop:comutWithSeries}, $[g_1\inv,D]$ is given by $\langle \set{c_1^{f(i)}\mid j\leq i\leq\Phi_0(j)}\rangle$. By definition of $f$, all numbers $f(i)$ for $3^j\leq 3^i\leq \tilde \Phi(3^j)$ are multiples of $p_j$, the first result follows. On the other hand, by \Cref{prop:gcd1}, $\gcd\{f(j),f(1+\log_3(\tilde \Phi(3^j)))\}=1$. It follows that $\langle[g_1,a_{3^j}],[g_1,a_{3\tilde \Phi(3^j)}]\rangle=c_1^\ZZ$ and thus that $g_2g_1\inv\in [g_1,D_j^+]$ as had to be shown.
	\end{proof}
	\begin{corollary}\label{prop:clLowerBound}
		Let $d$ be as above, then $\cl_{G,S}\succ \Phi$.
	\end{corollary}
	\begin{proof}
		Let $i$ be an integer, and assume $i=3^j$ is a power of $3$, we will show that $\cl_{G,S}(3i+8)> \tilde \Phi(i)$.
		Indeed, Let $g_1,g_2$ be as in \Cref{prop:conjWidthLowerBound}, notice that $g_1=(a_0t)^jt^{-j}$ is of word norm at most $3j$ and that $g_2=g_1[a_0,t\inv a_0t]$ is thus of word norm at most $3^{j+1}+8$. Furthermore by the second part of \Cref{prop:conjWidthLowerBound}, $g_1$ and $g_2$ are conjugate. Let $h\in G_{d_f}$ be such that $hg_1h\inv=g_2$. By \Cref{prop:Centraliser}, $h\in D$. Assume from contradiction that $\norm{h}_S\leq \tilde \Phi(3^j)$, then by \Cref{prop:support} we have that $h\in \langle S_\Phi(3^j)\rangle$. By the first part of \Cref{prop:conjWidthLowerBound}, such an $h$ does not exist.
	\end{proof}
	\begin{remark}
		Notice that the above lower bound also works in more generality if we replace $\Phi$ with a smaller function, however in that case the upper bound will no longer hold.
	\end{remark}
	To proceed with the upper bound, we first characterise when two elements of $D$ are conjugate.
	\begin{lemma}
		Let $g_1=c_1^{e_c}\prod a_i^{e_i}$ be some finite product and $g_2=g_1c_1^e$, then $g_1$ and $g_2$ are conjugate in $G_{d_f}$ if and only if $\lcm\{e_i\}\mid e$.
	\end{lemma}
	\begin{remark}Notice that every element of $N$ can be brought into the form of $g_1$ above.\end{remark}
	\begin{proof}
		As $c_1$ is central, we might assume that $e_c=0$.
		Let $h$ such that $hg_1h\inv=g_2$, by \Cref{prop:Centraliser}, $h\in N$. It thus suffices to show that $g_2g_1\inv\notin[g_1,N]$.
		
		First assume that $\lcm\{e_i\}\nmid e$.
		Let $k_j$ such that $c_1^{k_j}=[g_1,a_j]$. This commutator can be expanded as\begin{align*}
			c_1^{k_j}&=[\prod a_i^{e_i},a_j]\\
				 &=\prod [a_i,a_j]^{e_i}\\
				 &=c_1^{\sum e_i d(i-j)}\\
		\end{align*}
		This last exponent is a multiple of $\lcm\{e_i\}$. By bi-linearity, it follows that $[g_1,N]\subset c_1^{\ZZ\gcd\{e_i\}}$. As $c_1^{e}$ is not an element of this subgroup, it follows that $g_1\nsim g_2$.
		On the other hand, suppose that $\lcm\{e_i\}\mid e$, then there exist coefficients $k_i\in\ZZ$ such that $\sum k_i e_i=e$. On the other hand, Let $m$ be such that $g_1\in \langle S_m\rangle$, let $j$ be such that $3^j> 2m$, by \Cref{prop:gcd1}, there exist coefficients $l_1,l_2\in \ZZ$ such that $l_1f(j)+l_2f(1+\Phi_0(j))=1$. Define then $$
		h=\prod a_i^{k_i(l_1+l_2)} a_{i+3^j}^{-k_il_1} a_{i+3\tilde \Phi(3^j)}^{-k_il_2}.
		$$
		We claim $[g_1,h]=c_1^{e}$.
		Indeed by \Cref{prop:SelectorCommutator}, we might rewrite\begin{align*}
			&[\prod a_i^{e_i},\prod a_i^{k_i(l_1+l_2)} a_{i+3^j}^{-k_il_1} a_{i+3\tilde \Phi(3^j)}^{k_il_2} ]\\
			=&\prod_{i=-m}^m c_1^{e_ik_il_1f(j)+e_ik_il_2f(1+\Phi_0(j))}\\
			=&c_1^{\left(\sum e_ik_i\right)\left(l_1f(j)+l_2f(1+\Phi_0(j))\right)}\\
			=&c_1^{e}.
		\end{align*}
		which we had to demonstrate.
	\end{proof}
	To prove the upper bound, we must estimate the coefficients $k_i$ and $l_i$ from the above proof.
	The coefficients $k_i$ are already bounded by \Cref{prop:gcdintegers}. For the other coefficients $l_i$, we provide $2$ lemmas to distinguish between the $2$ cases \ref{cond:GlobalLarge} and \ref{cond:LocalLarge}: The first lemma works best when $\Phi$ satisfies \ref{cond:GlobalLarge}
	\begin{lemma}\label{prop:gcdSequenceLarge}
		Let $f$ be as in \Cref{def:dArbitraryConjugatorLength} and let $e<p_{j_0}^{j_0}3^{2j_0}$ for $j_0\in\NN$. Let $j=1+\Phi_0(j_0)$. Then there exist $l_{j_0},l_{j_0+1},\cdots,l_j$ such that $\sum \abs{l_i}\leq f(j_0)+2p_j(j-j_0)$ and such that $\sum l_if(i)=e$. 
	\end{lemma}
	\begin{proof}
		\sloppy
		By construction of $f$, There exists some finite sequence $x_i$ such that $x_{i+1}\mid x_i$ and such that $f(i)=x_i\prod_{k=j}^{i}p_k$ for every $i\in[j_0,j)$.
		By \Cref{prop:dLowerBound}, $f(j)\geq \abs{e}$. As by \Cref{prop:gcd1}, $\gcd\{f(j_0),f(j)\}=1$, there exists some constant $l_j$ such that $\abs{l_j}\leq f(j_0)$ and such that ${e-l_jf(j)\cong 0\mod f(j_0)}$.
		notice that $\abs{e-l_jf(j)}\leq 2f(j)f(j_0)$. Consider the sequence $y(i)=\frac{f(i)x(j_0)}{x(i)}$. Notice that $\frac{y(i)}{y(i-1)}=p_i$. Notice that $2p_jy(j-1)\geq f(j_0)f(j)$. We can thus find integers $l'_i\in\ZZ$ such that $\abs{l'_i}\leq p_{i+1}$ for $i<j$ and $\abs{l'_{j-1}}\leq 2p_j$ and such that $\sum_{i=j_0}^{j-1} l'_iy(i)=e-l_jf(j)$. If we then define $l_i=l'_i\frac{x(j_0)}{x(i)}$, then we have $\sum_{i=j_0}^j l_if(i)=e$. Furthermore, from each of the individual bounds on $l_i$, we obtain that $\sum \abs{l_i}\leq f(j_0)+2p_j(j-j_0)$.
	\end{proof}
	We change our setup for the case when $\Phi$ satisfies \ref{cond:LocalLarge}.
	\begin{lemma}\label{prop:nIncreasing}
		Let $\Phi$ satisfy \ref{cond:LocalLarge}, then $f$ is increasing.
	\end{lemma}
	\begin{proof}
		Let $p_i$ divide $f(j)$, then must $3^i\leq3^j\leq \Phi(3^i)$. It follows that $3^{i+1}\leq 3^{j+1}\leq \Phi(3^{i+1})$, or thus must $p_{i+1}$ still divide $f(j+1)$. In particular, $f(j+1)$ might lose at most one prime factor compared to $f(j)$, and this prime factor has index $i<j+1$. It follows that $$f(j+1)\geq f(j)\frac{p_j}{p_i}\geq f(j).$$
	\end{proof}
	\begin{lemma}\label{prop:gcdSequenceSmall}
		Let $f$ be as before, let $\Phi$ satisfy \ref{cond:LocalLarge} and let $j_0\in \NN_{>2}$. Let $e<p_{j_0}^{j_0}3^{2j_0}$ and let $j=1+\Phi_0(i)$. Then there exist $l_{j_0},l_{j_0+1},\cdots,l_j$ such that $\sum \abs{l_i}\leq jp_j$ and such that $\sum l_if(i)=e$. 
	\end{lemma}
	\begin{proof}
		\sloppy
		\newcommand{\emax}{{e_{\mathrm{max}}}}
		Let $\emax$ denote $p_{j_0}^{j_0}3^{2j_0}$
		Notice that by \Cref{prop:nIncreasing}, that $f$ is increasing.
		We first prove inductively that for any $j_0\leq j'<j$ there exist constants $l_{i}$ for $j_0\leq i< j'$ such that $l_i\leq p_{i+1}$, $$\prod_{i=j_0+1}^{j'}p_i\mid e+\sum _{i=j_0}^{j'-1}l_{i} f(i)$$ and $$
		\abs{e+\sum _{i=j_0}^{j'-1}l_{i} f(i)}\leq \max\{p_{j'}f(j'-1),\emax\}.
		$$
		For the base step, when $j'=j_0$ the result holds by the assumptions. Suppose the result holds for $j'$, let then $l_i$ be such that the result holds and assume without loss of generality that ${e_{j'}=e+\sum _{i=j_0}^{j'-1}l_{i} f(i)}$ is positive. Let now $l_{j'}\in\ZZ_{\leq 0}$ be maximal (minimal in absolute value)  such that ${e_{j'}+l_{j'}f(j')\cong 0\mod p_{j'+1}}$. Such an integer always exists as $p_{j'+1}$ does not divide $f(j')$, and it can be chosen such that $\abs{l_{j'}}\leq p_{j'+1}$. By construction of $l_{j'}$, we have that $p_{j'+1}$ divides $e+\sum _{i=j_0}^{j'}l_{i} f(i)=e_{j'+1}$ and as $\prod_{i=j_0+1}^{j'}p$ divides both $e_{j'}$ and $f(j')$, it must also divide $e_{j'+1}$. Finally we have $\max\{p_{j'}f(j'-1),\emax\}\geq e_{j'+1}\geq -p_{j'+1}f(j')$, implying that $\abs{e_{j'+1}}\leq\max\{p_{j'+1}f(j'),\emax\}$. This completes the induction argument.
		
		We thus have that $e_j$ is a multiple of $\prod_{i=j_0+1}^{j}p_i$ that is at most $\max\{p_{j}f(j-1),\emax\}$. By \Cref{prop:dLowerBound}, this is less then $p_jf(j)$.
		We can thus find $l_j$ with $\abs{l_j}\leq p_j$ such that ${e+\sum _{i=j_0}^{j-1}l_{i} f(i)+l_j f(j)=0}$.
		The result follows as $\sum \abs{l_j}\leq jp_j.$		
	\end{proof}
	
	We now conclude our proof on the conjugator growth:
	\begin{proposition}
		Let $f$ be as before, then $\cl_{G_{d_f}}\prec \Phi$.
	\end{proposition}
	\begin{proof}
		Let $g_1,g_2\in G_{d_f}$.
		By \Cref{prop:conjugatorZwrZ}, we might assume that $g_1C=g_2C$. If $g_1,g_2\notin N$, then by \Cref{prop:Centraliser} we are done. We might thus assume $g_1,g_2\in N$. Suppose that $\norm{g_1}_S,\norm{g_2}_S\leq 3^{j_0-1}$. Then $g_2g_1\inv$ is of norm at most $2\cdot 3^{j_0-1}$. In particular by \Cref{prop:supportC}, $g_2g_1\inv=c_1^{e}$ where $e\leq 4\cdot 3^{2j_0-2}f(j_0)$.
		We might write $g_1=c_1^{e_c}\prod a_i^{e_i}$. By \Cref{prop:support}, the product runs over indices $-m\leq i\leq m$ and $\sum \abs{e_i}\leq m$. 
		Let $\tilde e$ be the greatest common divisor or the exponents $e_i$. By \Cref{prop:gcdintegers}, there exist coefficients $k_i$ such that $\sum{e_ik_i}=\tilde e$. Furthermore, let $j=1+\Phi_0(j_0)$, then by \Cref{prop:gcdSequenceLarge} or \Cref{prop:gcdSequenceSmall}, there exist integers $l_i$ where the indices run from $j_0$ to $j$ such that $\sum \abs{l_i}\leq f(j_0)+2p_j(j-j_0)$ in case of \ref{cond:GlobalLarge} or $\sum \abs{l_i}\leq jp_j$ in case of \ref{cond:LocalLarge}.
		
		Let now $h$ be given by $$
		h=\prod_{i'=j_0}^j\prod_{i=-m}^m a_{i}^{l_{i'}k_i}a_{i+3^{i'}}^{-l_{i'}k_i}.
		$$
		First we demonstrate that $[g_1,h]=c_1^{e}$. Indeed, by \Cref{prop:SelectorCommutator}, we have\begin{align*}
			&[\prod a_i^{e_i},\prod_{i'=j_0}^j\prod_{i=-m}^m a_{i}^{l_{i'}k_i}a_{i+3^{i'}}^{-l_{i'}k_i}]\\
			=&\prod_{i'=j_0}^j\prod_{i=-m}^m c_1^{e_{i}k_{i}(l_{i'} f(i'))}\\
			=&c_1^{\sum_{i'=j_0}^j l_{i'} f(i')\sum_{i=-m}^m(e_ik_i)}\\
			=&c_1^{\frac{e}{\tilde e}\tilde e}
		\end{align*}
		With this we have shown that $h\inv g_1h=g_2$. Now it is left to demonstrate that there exists some $c\in C$ such that $\norm{ch}_S$ is not too large compared to $\Phi(3^{j_0})$. 
		For this, we first compute $\norm{h}_{S_{3^j}}$. This norm is bounded above by $2\sum \abs{l_i}\sum\abs{k_i}$. The second of these factors is bounded by $3^{j_0-1}$.
		
		The computation of the first depends on the $2$ cases for $\Phi$.
		In case \ref{cond:LocalLarge}, we have an upper bound of $\sum \abs{l_i}\leq jp_j$. Notice that $p_j$ is at most $2j^2<C3^{\frac{j}{2}}<c\frac{\Phi(3^{j_0})}{3^{j_0}}$.
		In the second case, this factor is bounded above by $f(j_0)+2p_j(j-j_0)$. $p_j$ is again at most $2j^2$. On the other hand $f(j_0)$ is by \Cref{prop:nUpperBound} at most $p_{j_0}^{j_0}$ which can be bounded by $(3^{j_0})^{\log_3(2)+\log_3\log_3(3^{j_0})}$. By the assumption on $\Phi$, this is less then $\frac{\Phi(3^{j_0})}{3^{j_0}}$.
		
		In both of these cases, we obtain that $2\sum \abs{l_i}\sum\abs{k_i}$ is bounded by $2c\Phi(3^{j_0})$. We now use \Cref{prop:supportInverse} to find some $h'$ such that $hC=h'C$ and such that $\norm{h'}\leq (2c+4)\Phi_0(3^{j_0})$.
		
		The result follows as $[g_1,h']=[g_1,h]$.
	\end{proof}
	With the above proposition, and with \Cref{prop:clLowerBound}, we obtain the main result of our section:
	\setcounter{result}{2}
	\begin{result}\label{prop:clSpectrum}
		Let $\Phi:\NN\rightarrow\NN$ be a function satisfying \ref{cond:LocalLarge} or \ref{cond:GlobalLarge}. Then there exists some torsion-free, finitely generated, $3$-step solvable group such that $\cl_G\simeq \Phi$. Furthermore if $\Phi$ is computable, then $G$ can be chosen to be recursively presented.
	\end{result}
	\change{question 1 was moved to new subsection}
	
	\subsection{Superlinear and subquadratic conjugator length}\label{sec:Linear/Quadratic}\change{new subsection}
	In the previous \namecref{sec:conjugatorLength}, the focus lied on obtaining large functions as $\cl_G$. 
	In this \namecref{sec:Linear/Quadratic}, we complement it by giving a group $G$ where the conjugator length function is superlinear but subquadratic. In particular, we construct a group where $n^{\frac{4}{3}}\prec\cl_{G}(n)\prec n^{\frac{5}{3}}$.

	Let $m_i$ be the sequence, defined inductively by $m_0=6$ and $m_{i+1}=\lfloor(m_i)^{\frac{3}{2}}\rfloor$. We define the second sequence $\tilde m_i=\lfloor\frac{m_i}{2}\rfloor$. Notice that these two sequences have disjoint images.\footnote{6 could be replaced by a larger integer.}
	
	We now define the function $d:\ZZ\rightarrow\ZZ$ by $$
	d(n)=\begin{cases}
		\begin{matrix}
			n&\exists i:\abs{n}=m_i\\
			\mathrm{sign}(n)&\exists i:\abs{n}=\tilde m_i \mathrm{or} \abs{n}=1\\
			0&\text{otherwise}
		\end{matrix}
	\end{cases}
	$$
	In the previous \namecref{sec:conjugatorLength}, the arithmetic behaviour of the functions $d_f$ played an important role, here however only the growth characteristics of $d$ are of importance.
	\begin{proposition}\label{prop:clIntermediateLower}
		$\cl_{G_d}(n)\succ n^{\frac{4}{3}}$
	\end{proposition}
	\begin{proof}
		Let $n>0$ be sufficiently large, and let $I$ be maximal such that $m_I\leq n$. Then the element $g_2=a_0c_1^{n^2m_I}=a_0[a_0^n,a_{m_I}^n]$ is of word norm at most $4m_I+4n+1\leq 8n+1$ with respect to $S$. 
		Furthermore this element is conjugate to $g_1=a_0$ as is exemplified by $g_2=a_{m_I}^{-n^2}g_1a_{m_I}^{n^2}$. Suppose that $h$ with $\norm{h}_S\leq \frac{n^{\frac{4}{3}}}{2}$ is such that $hg_1h\inv=g_2$. First because of \Cref{prop:Centraliser}, $h\in N$.
		Then, applying \Cref{prop:support}, we have that $h$ must be given, up to some element of $C$ by $$
		\prod_{i=-\lfloor\frac{n^{\frac{4}{3}}}{2}\rfloor}^{\lfloor\frac{n^{\frac{4}{3}}}{2}\rfloor} a_i^{e_i}
		$$
		where $\sum\abs{e_i}\leq \frac{n^{\frac{4}{3}}}{2}$. We may now estimate the exponent of $[g_1,h]=c_1^{e}$ as $e=\sum d(i)e_i\leq\frac{n^{\frac{8}{3}}}{4}$. As $n^2m_I$ is at least $n^2(n^{2/3}-1)$, it follows that $hg_1h\inv\neq g_2$ and thus from contraposition it follows that $\cl_{G_d,S}(8n+1)> \frac{n^{\frac{4}{3}}}{2}\succ n^{\frac{4}{3}}$.
	\end{proof}
	For the upper bound we first demonstrate the following \nameCref{prop:spacingOfSequence} on the sparsity of $d$.
	\begin{lemma}\label{prop:spacingOfSequence}
		For all $\epsilon>0$ there exists some $N_0>0$ such that if $n>N_0$ and $m_I>n$, then $\tilde m_{I+1}-m_{I}>m_{I}- \tilde m_{I}>\tilde m_{I}- m_{I-1}>\frac{n}{2+\epsilon}$
	\end{lemma}
	\begin{proof}
		As soon as $m_{I-1}>4$, all the inequalities except potentially the last hold automatically. Asymptotically, $\tilde m_{I}- m_{I-1}$ grows like $m_I(\frac{1}{2}-m_I^{-\frac{2}{3}})$. It follows that for $m_I$ sufficiently large that $\tilde m_{I}- m_{I-1}>m_I\frac{1}{2+\epsilon}$ implying the result.
		\end{proof}
	\begin{proposition}\label{prop:clIntermediateUpper}
		$\cl_{G_d}(n)\prec n^{\frac{5}{3}}$
	\end{proposition}
	\begin{proof}
		Let $n$ be arbitrary and let $g_1,g_2$ be such that $\norm{g_1}_S,\norm{g_2}_S\leq n$. First, by \cref{prop:conjugatorZwrZ}, we may assume that $g_1C=g_2C$, increasing $n$ by at most a constant factor in the process.
		Now by \Cref{prop:conjInSubgroup}, we may assume that $g_1,g_2\in N$. Furthermore, by the same \namecref{prop:conjInSubgroup}, it suffices to find some $h$ of sufficiently short word norm such that $g_2g_1\inv=[h,g_1]$
		
		Let $I$ be maximal such that $m_I\leq 5n$. By \Cref{prop:support}, we can rewrite $g_1C$ as $$
		\prod_{i=-n}^na_i^{e_i}C
		$$
		where $\sum \abs{e_i}\leq n$ and by \Cref{prop:supportC} we may rewrite $g_2g_1\inv$ as $c_1^e$ where $e\leq 4n^2m_I$.
		Let $\tilde e$ be the greatest common divisor of the integers $e_i$. Clearly, $[N,g_1]$ is contained in $\langle c_1^{\tilde e}\rangle$, thus by \Cref{prop:conjInSubgroup}, $\tilde e$ must divide $e$. By \cref{prop:spacingOfSequence} we have for $n$ sufficiently large that $\tilde m_{I+2}-m_{I+1}>m_{I+1}- \tilde m_{I+1}>\tilde m_{I+1}- m_{I}>2n+1$.
		In particular, we have for $j\in[-n,n]$ that $[g_1,a_{j+\tilde m_{I+1}}]=c_1^{e_j}$ and $[g_1,a_{j+{m_{I+1}}}]=c_1^{e_jm_{I+1}}$. Let $j_0$ be such that $e_{j_0}$ is minimal among the non-$0$ entries, we can now find $\abs{\lambda_1}\leq\frac{e}{e_{j_0}m_{I+1}}\leq 4n^2\frac{m_I}{m_{I+1}}$ and $\abs{\lambda_2}<m_{I+1}$ such that $e'=e-\lambda_1e_{j_0}m_{I+1}-\lambda_2e_{j_0}$ satisfies $ e'\leq e_{j_0}$.
		Notice that $\tilde e\mid e'$. By \Cref{prop:gcdintegers}, there exist integers $k_{-n},\cdots k_n$ such that $\sum \abs{k_i}\leq \sum\abs{e_i}$ and such that $\sum k_ie_i=e'$.
		Let $h$ now be given by $$h= a_{j_0+ m_{I+1}}^{-\lambda_1}a_{j_0+\tilde m_{I+1}}^{-\lambda_2}\prod_{i=-n}^n a_{i+\tilde m_{I+1}}^{-k_i}.$$
		We compute the commutator $[h,g_1]=c_1^{m_{I+1}\lambda_1+\lambda_2 +\sum k_ie_i}=c_1^e$.
		By \Cref{prop:supportInverse} we may now find $h'$ with $h'C=hC$ such that 
		$$\norm{h'}\leq 4(n+m_{I+1}) + n+4n^2\frac{m_I}{m_{I+1}}+m_{I+1}.$$
		As $N$ is $2$-step nilpotent, we thus also have $[h',g_1]=g_2g_1\inv$.
		The result now follows as $m_{I+1}\leq (6n)^{\frac{3}{2}}$ and $\frac{m_I}{m_{I+1}}\leq (6n)^{\frac{-1}{3}}$.
	\end{proof}
	
	\begin{question}
		Which other functions appear as conjugator length functions?
		Can $\cl_G$ be sublinear?
	\end{question}
	\section{Intermediate growth}\label{sec:intermediate}
	In this section we demonstrate that there exists a group with intermediate residual finiteness growth. We will need some sequences of specific growth types, for this we first recall compositional roots.
		\subsection{Compositional roots}
	In what follows we make use of \textbf{compositional $n$-th roots} of the exponential function. That is a function $\mf f:\RR\rightarrow \RR$ such that $\mf f^n(x)=\exp(x)$. Here $f^n(x)$ denotes $f\comp f^{n-1}$ where $f^1(x)=f(x)$. For (multiplicative) powers of a function $f$, we will include the argument and denote $f(x)^n$.
	Such roots exist and can even assumed to be monotone. 
	We give a classical way of constructing such $n$-th roots of $\exp$:
	
	Let $0=c_0<c_1<c_2<\cdots<c_{n-1}<c_n=1$ be constants. For $0\leq i\leq n-2$, define $\mf f_0:[c_i,c_{i+1}]\rightarrow\RR$ as the linear function $$x\mapsto c_{i+1}+\frac{c_{i+2}-c_{i+1}}{c_{i+1}-c_i}(x-c_i).$$
	And on $[c_{n-1},c_n]$ define $\mf f_0$ as $$x\mapsto \exp(c_{1}+\frac{c_{1}-c_{0}}{c_{n}-c_{n-1}}(x-c_{n-1})).$$
	
	Now on any interval $[\exp^k(0),\exp^k(1)]$, define $\mf f$ as $$\exp^k\comp \mf f_0\comp \log^k.$$ One checks that $\mf f$ is continuous. As both $\mf f_0,\exp$ and $\log$ are strictly increasing, $\mf f$ is also strictly increasing. One checks that indeed $\mf f^n(x)=\exp(n)$ for any $x$.
	
	The functions constructed above, are of intermediate type:
	\begin{lemma}\label{prop:fIsIntermediate}
		Let $\mf f$ be a monotone $n$-th compositional root of $\exp$ for $n>1$, then for any polynomial $p(x)$, there exists some $x_0$ such that $\forall x\geq x_0:\:\mf f(x)>p(x)$. And for every $\epsilon>0$, there exists some $x_0$ such that $\forall x\geq x_0:\:\mf f^{n-1}(x)<e^{x^\epsilon}$.
	\end{lemma}
	\begin{proof}
		As $\exp(x)>x$ and as $\mf f$ is monotone, we have that $\mf f(x)>x$. It follows that on $[0,1]$, $\mf f(x)>x+\epsilon$ for some $\epsilon>0$. Suppose $\mf f(x)>x+\epsilon$ holds on some interval $[0,y]$ with $y\geq 0$, we show that it then also must hold on $[0,\exp(y)]$. This is indeed the case as for $x\in [y,\exp(y)]$, $\log(x)\in [0,y]$ and thus $$\mf f(x)=(\exp\comp \mf f\comp\log)(x)\geq\exp(\log(x)+\epsilon)=x\exp(\epsilon)>x(1+\epsilon)\geq x+\epsilon.$$ By induction, $\mf f(x)>x+\epsilon$ for all positive $x$. Now for $x>\exp(e)$ we find that $$\mf f(x)=(\exp^3\comp \mf f\comp \log^3)(x)\geq x^{\log(x)^\epsilon}$$ which grows faster than any arbitrary polynomial.
		
		For the second bound, let $m$ be such that $\frac{1}{m}<\epsilon$, we have for $x$ sufficiently large that $$
		\mf f(x)>x^m
		$$
		and thus by applying $\mf f$ $n-1$ times that$$
		\mf f^{n-1}(x)>\exp(x^m).
		$$
		The result follows for substituting $y=\sqrt[m]{x}$.
	\end{proof}

	\subsection{Constructing the group}
	As a notational shorthand, henceforth intervals are to be interpreted as subsets of integers\footnote{i.e. whenever $a<b$, $[a,b]=\{i\in\ZZ\mid a\leq i\leq b\}.$}
	
	To construct our exact group, we need some auxiliary sequences.
	First fix $\mf f$ a monotonously-increasing compositional $10$-th root of $\exp$, let $\tilde d_n$ denote the sequence such that $\tilde d_0=2\cdot 3\cdot 5$ and $\tilde d_n=e^{\tilde d_{n-1}}$. And denote $\tilde P_n$ the sequence such that $\tilde P_0=\mf f(d_0)$ and such that $\tilde P_n=\mf f^2(\tilde P_{n-1})$. In particular we have $\mf f(\tilde d_n)=\tilde P_{5n}$ and $\mf f(\tilde P_{5n-1})=d_n$.
		
		We approximate each of these sequences with integer sequences with desired numerical properties. 
		Let $P_n$ be prime valued such that $\tilde P_n\leq P_n<2 \tilde P_n$. We define $d_n$ inductively with $d_0=2\cdot 3\cdot 5$ such that for all $n\geq 1$, $\tilde d_n\leq d_n<\tilde d_n^{\frac{4}{3}}$, and such that $\frac{d_n}{d_{n-1}}$ is the product of the primes up to $m_n$ for some integer $m_n$.
		\begin{lemma}
			Such a sequence $d_n$ exists.
		\end{lemma}
		\begin{proof}
			First notice that $\tilde d_n>25$ for all $n\geq 0$. We show the product of all primes $p\leq e^{\frac{d_n}{3}}$ is at least $e^{d_n}$. Note that $e^{\frac{d_n}{6}}\geq 64$. By Bertrands postulate we find at least $6$ distinct primes in the interval $[e^{\frac{d_n}{6}},e^{\frac{d_n}{3}}]$, the product of which is thus at least $e^{d_n}$. We find $d_{n+1}$ by starting with $d=d_n$ and successively multiplying it with larger and larger primes $p$. If $p<e^{\frac{d_n}{3}}$ and $d<e^{d_n}$, then $pd<e^{\frac{4d_n}{3}}$. By the above computation, we eventually have for some $p\leq e^{\frac{d_n}{3}}$ that $d<e^{d_n}\leq pd<e^{\frac{4d_n}{3}}$. We continue inductively.
		\end{proof}
		
		\begin{lemma}
			For $n\neq m$ sufficiently large, $P_n\neq P_m$
		\end{lemma}
		\begin{proof}
			As by \Cref{prop:fIsIntermediate}, $\mf f$ outgrows every polynomial, we eventually have that $2\tilde P_i<\tilde P_{i+1}$. The result follows by Bertrands postulate.
		\end{proof}
		We define now $$
		q_n=\prod_{k=0}^8 P_{5n-5+2k}
		$$
		where we take $P_i=3$ whenever $i<0$.
		
		For sufficiently large $n$, $q_n$ is thus composed of those primes $P_j$ appearing between $d_{n-1}$ and $\mf f^4(d_{n+2})$ where $j$ has different parity then $n$.

		Define $G_{\textrm{Int}}$ as the group with presentation
			$$
			G_{\textrm{Int}}=\lelangle \{t,a_i,c_i\vert i\in\ZZ\}\Bigg\vert \begin{matrix} ta_it\inv=a_{i+1}&i\in\ZZ\\
				tc_it\inv=c_i&i\in\ZZ\\
				[a_i,a_j]=c_{i-j}&i,j\in\ZZ\\ [c_i,a_j]=[c_i,c_j]=1&i,j\in\ZZ\\ c_{d_i}^{q_i}=1&i\in\ZZ_{\geq 0}
			\end{matrix}\rerangle.
			$$
		From the presentation it is clear that $G_{\textrm{Int}}\in\cat H$
		Notice that the centre of $G_{\textrm{Int}}$ is in a way larger than the centre of $G_d$ from the previous section as this centre contains a subgroup isomorphic to $\bigoplus_\NN\ZZ$.
		\begin{proposition}
			$\mathrm{Rf}_{G_{\textrm{Int}}}(n)\succ {\mf f}^2(n)$
		\end{proposition}
		\begin{proof}
			First notice that for $i\geq I_0$, some uniform constant, $P_i P_{i+1} P_{i+2}\cdots P_{j-1}\leq P_{j}$. Indeed on the one hand $P_i\leq P_{i+1}$, by induction it thus suffices to show that $P_i^2<P_{i+1}$ this is the case as by \Cref{prop:fIsIntermediate}, ${\mf f}^2(\frac{X}{2})$ eventually outgrows $X^2$.
			
			Let $n$ be a sufficiently large natural number, let $i$ be maximal such that $P_i\leq n$ and let $d_j\leq n$ be maximal such that $P_i\mid q_j$. Notice that $j$ is thus the either the maximal integer such that $5j\leq i$ or one less, whichever has the opposite parity of $i$.
			
			Let $I=P_{5j-5}P_{5j-3}\cdots P_i$, by the above $I<P_i^2\leq n^2$. Notice that $x=c_{d_j}^{I}\in{G_{\textrm{Int}}}$ is of word norm no more then $10n+2$ as it can be written as $$
			[a_0^{m_0},a_j^n][a_0^{m_1},a_j]
			$$
			for some $\abs{m_0},\abs{m_1}\leq n$. The order of $c_j$ is $\prod_{k=0}^8 P_{5j-5+2k}$ thus the order of $x$ is $$
			P_{i+2}P_{i+4}\cdots P_{5j+11}.
			$$
			Let $\varphi:{G_{\textrm{Int}}}\rightarrow Q$ be a finite quotient in which $\varphi(x)$ is non-trivial. The order of $\varphi(x)$ must be a divisor of the order of $x$, so it must be at least $P_{i+2}$, furthermore $P_{i+2}\geq {\mf f}^2(\tilde P_{i+1})\geq {\mf f}^2(\frac{n}{2})$ from which the result follows.
		\end{proof}
		\begin{lemma}\label{prop:centerreduct}
			Fix $\epsilon>0$ some constant.
			Let $n$ be an integer.
			Let ${0\leq l_1<l_2<\cdots<l_k\leq n}$ and $m_0$ be integers such that $l_i\neq l_j\mod m_0$.
			Let $f\in \ZZ[X]$ be of degree $\leq n$ such that ${f\notin M=\ZZ X^{l_1}+\ZZ X^{l_2}+\cdots+ \ZZ X^{l_k}}$.
			Then there exists some odd prime $q<cn^{\frac{1}{2}+\epsilon}$ for some uniform constant $c$ such that for the ideal $I_q=\langle X^{\mathrm{lcm}(q,m_0)}-1\rangle$, $f\notin M+I_q$.
		\end{lemma}
		\begin{proof}
			By the prime density theorem, we have that $\sum_{\substack{q<n\\q\textrm{ prime}}}q\succ \frac{n^2}{\log(n)^2}$. We might thus find some constant $c$ be such that $$\sum_{\substack{q<cn^{\frac{1}{2}+\epsilon}\\q\textrm{ prime}\\q>2}}q-1>n.$$ 
			Denote $\mc P=\{q\textrm{ prime}\mid q>2, q<cn^{\frac{1}{2}+\epsilon}\}$.
			
			Suppose that for every $q\in\mc P$, $f\in M+I_q$. Then there exist constants $c_{i,q}$ such that $$f=c_{1,q}X^{l_1}+c_{2,q}X^{l_2}+\cdots+c_{k,q}X^{l_k}\mod I_q.$$
			Let $q_0$ be a prime divisor of $m_0$, notice that $I_{q_0}\supset I_{q}$ for every $q$, in particular we have for any $q$ that
			$$f=c_{1,q}X^{l_1}+c_{2,q}X^{l_2}+\cdots+c_{k,q}X^{l_k}\mod I_{q_0}.$$
			
			Notice that the natural projection $\ZZ[X]\rightarrow \frac{\ZZ[X]}{I_{q_0}}$ is injective on $M$. It follows that if \begin{align*}
				&c_{1,q}X^{l_1}+c_{2,q}X^{l_2}+\cdots+c_{k,q}X^{l_k}\\
				&=c_{1,q'}X^{l_1}+c_{2,q'}X^{l_2}+\cdots+c_{k,q'}X^{l_k}\mod I_{p_0},
			\end{align*}
			then $c_{i,q}=c_{i,q'}$ for any choice of $i$, we may thus denote $c_i=c_{i,q}$.
			We find by the Chinese remainder theorem that $$
			f=c_{1}X^{l_1}+c_{2}X^{l_2}+\cdots+c_{k}X^{l_k}\mod\bigcap_{q\in\mc P}I_{q}.
			$$
			This ideal $$\bigcap_{q\in\mc P}I_{q}$$ is contained in the principle ideal generated by the product of cyclotomic polynomials $$\prod_{q\in\mc P} \phi_q(X)$$ of degree $$\sum_{q\in\mc P}q-1.$$
			As both $f$ and $c_{1}X^{l_1}+c_{2}X^{l_2}+\cdots+c_{k}X^{l_k}$ are of degree at most $$n<\sum_{\substack{2<q<Cn^{\frac{1}{2}+\epsilon}\\q\textrm{ prime}}}q-1,$$ we have that they must thus be equal, contradicting that $f\notin M$.
		\end{proof}
		We can translate \Cref{prop:centerreduct} to a version of Laurent polynomials:
		
		\begin{lemma}\label{prop:laurentCenterReduct}
			Fix $\epsilon>0$ some constant.
			Let $n>0$ be an integer.
			Let ${-n\leq l_1<l_2<\cdots<l_k\leq n}$ and $m_0$ be integers such that $l_i\neq l_j\mod m_0$.
			Let $f\in \ZZ[X,X\inv]$ be a Laurent-polynomial where each term has degree in $[-n,n]$ such that ${f\notin M=\ZZ X^{l_1}+\ZZ X^{l_2}+\cdots+ \ZZ X^{l_k}}$.
			Then there exists some odd prime $q<cn^{\frac{1}{2}+\epsilon}$ for some uniform constant $c$ such that for the ideal $I_q=\langle X^{\mathrm{lcm}(q,m_0)}-1\rangle$, $f\notin M+I_q$.
		\end{lemma}
		\begin{proof}
			Let $\tilde f=X^nf\in\ZZ[X]$, let $\tilde l_i=n+l_i\in[0,2n]$ and define $\tilde M=X^nM\subset\ZZ[X]$. By \Cref{prop:centerreduct} we find $q<c(2n)^{\frac{1}{2}+\epsilon}$ such that for the ideal $\langle X^{\mathrm{lcm}(q,m_0)}-1\rangle\tilde I_q<\ZZ[x]$, we have $\tilde f\notin \tilde M+\tilde I_q$. As in $\ZZ[X]$, $x\nmid X^{\mathrm{lcm}(q,m_0)}-1 $, we have that $\tilde I_q=I_q\cap\ZZ[X]$. We thus have that $X^nf\notin X^nM+I_q$ of thus as $I_q$ is an ideal and $X$ a unit that $f\notin M+I_q$.
		\end{proof}

		\begin{lemma}\label{prop:dManyFactors}
			Let $c$ be some fixed constant, then for $n$ sufficiently large, any prime $p\leq c\sqrt{d_n}$ divides $d_{n+1}$.
		\end{lemma}
	 	\begin{proof}
	 		By construction of $d_i$, we have to demonstrate that $$
	 		d_n\prod_{p<c\sqrt{d_n}}p\leq d_{n+1}.
	 		$$
	 		This is the case as the former can be bounded above by 
	 		$$
	 		\tilde d_n^{\frac{4}{3}} (c\tilde d_n^{\frac{2}{3}})^{c\tilde d_n^{\frac{2}{3}}}
	 		$$
	 		which is less than $e^{\tilde d_n}$ for $n$ sufficiently large.
	 	\end{proof}
		\begin{lemma}\label{prop:bound2power}
			Let $n,k>2$, if $2^{k}\mid d_{n}$, then $2^k<2\log(\log(d_n))$. 
		\end{lemma}
		\begin{proof}
			Notice that by construction, if $2^k\mid d_{n}$, then $2^{k-2}\mid d_{n-2}$. As ${3\cdot 5\mid d_0}$, we have ${3\cdot 5\cdot 2^{k-2}\mid d_{n-2}}$. And thus $2^k\leq 3\cdot 5\cdot 2^{k-2}\leq d_{n-2}\leq 2\log (\log(\tilde d_n))$.
		\end{proof}
		To provide an upper bound to to the residual finiteness growth, we use quotients of the following form
		\begin{definition}
			Let $q>1$ and $p$ an odd prime, define then $G_{p,q}$ the group with presentation
			$$
			G_{p,q}=\lelangle \{t,a_i,c_i\vert i\in\frac{\ZZ}{2q\ZZ}\}\Bigg\vert \begin{matrix} ta_it\inv=a_{i+1}&i\in\frac{\ZZ}{2q\ZZ}\\
				tc_it\inv=c_i&i\in\frac{\ZZ}{2q\ZZ}\\
				[a_i,a_j]=c_{i-j}&i,j\in\frac{\ZZ}{2q\ZZ}\\ [c_i,a_j]=[c_i,c_j]=1&i,j\in\frac{\ZZ}{2q\ZZ}\\ c_{d_i}^{q_i}=1&i\in\ZZ_{\geq 0}\\
				t^{2q}=a_i^p=c_i^p=1&i\in\frac{\ZZ}{2q\ZZ}
			\end{matrix}\rerangle
			$$
			and $N_{p,q}$ with presentation
			$$
			N_{p,q}=\lelangle \{a_i,c_i\vert i\in\frac{\ZZ}{2q\ZZ}\}\Bigg\vert \begin{matrix} 
				[a_i,a_j]=c_{i-j}&i,j\in\frac{\ZZ}{2q\ZZ}\\ [c_i,a_j]=[c_i,c_j]=1&i,j\in\frac{\ZZ}{2q\ZZ}\\ c_{d_i}^{q_i}=1&i\in\ZZ_{\geq 0}\\
				a_i^p=c_i^p=1&i\in\frac{\ZZ}{2q\ZZ}
			\end{matrix}\rerangle.
			$$
		\end{definition}
		Notice that just as was the case with $G_0$, $c_i$ and $c_{-i}\inv$ get identified.
		Clearly $G_{p,q}$ is a quotient of ${G_{\textrm{Int}}}$ where $t,a_i,c_i$ get mapped to $t,a_i,c_i$ respectively, denote $\varphi_{p,q}$ the quotient map.
		Notice $G_{p,q}$ is a semi-direct product $N_{p,q}\rtimes\frac{\ZZ}{2q\ZZ}$ where the second group acts trivially on the generators $c_i$, and by index shift on the generators $a_i$.
		Notice furthermore that $N_{p,q}$ is $2$-step nilpotent where the derived subgroup $C_{p,q}$ is generated by $\{c_i\mid i\in\frac{\ZZ}{2q\ZZ}\}$.
		\begin{lemma}\label{prop:order}
			The order of $G_{p,q}$ is less than $2qp^{4q}$.
		\end{lemma}
		\begin{proof}
			A Malcev basis for $N_{p,q}$ is given by (some subset of the) set $\{a_i,c_i\mid i\in \frac{\ZZ}{2q\ZZ}\}$. All these elements are of order at most $p$ so $N_{p,q}$ is of order at most $p^{4q}$. The result follows as $$G_{p,q}=N_{p,q}\rtimes\frac{\ZZ}{2q\ZZ}.$$
		\end{proof}
		
		We are mainly interested in elements of $C<{G_{\textrm{Int}}}$, we thus describe $C_{p,q}$.
		\begin{lemma}\label{prop:QuotientOfC}
			The group $C_{p,q}$ is given by the free $\frac{\ZZ}{p\ZZ}$-module with basis$$
			\{c_i\mid i\in [1,q-1] \forall j\in\ZZ_{\geq 0}:p\nmid q_j\Rightarrow i\ncong \pm d_j\mod 2q\}.
			$$
		\end{lemma}
		\begin{proof}
			Notice that all generators $c_i$ with $i\in[q+1,2q-1]$ can be identified with $-c_{2q-i}$.
			We use multiplicative notation in this proof.
			Suppose that $p\mid q_j$, then the relation $c_{d_j}^{q_j}$ becomes redundant. Suppose on the other hand that $p\nmid q_j$, then we have $c_{d_j}^p=1=c_{d_j}^{q_j}$ for $p$ and $q_j$ coprime. It follows that $c_{d_j}=1$. 
		\end{proof}
		We define a closely related module.
		\begin{definition}
			Denote $\tilde C_{p,q}$ the free $\frac{\ZZ}{p\ZZ}$-module with basis $$
			\{c_i\mid i\in [-q+1,-1]\cup[1,q-1] \forall j\in\ZZ_{\geq 0}:p\nmid q_j\Rightarrow i\ncong \pm d_j\mod 2q\}.
			$$
			Denote $\delta:C_{p,q}\rightarrow\tilde C_{p,q}$ the morphism mapping $c_i$ to $c_i-c_{-i}$.
		\end{definition}
		For notational convenience, if $c_i$ is not in the above basis, then we take $c_i=0$.
		
		\begin{proposition}The group ${G_{\textrm{Int}}}$ is residually finite and for any $\epsilon>0$, $\mathrm{Rf}_{{G_{\textrm{Int}}},S}(n)\prec e^{x^{\frac{1}{2}+\epsilon}}$.
		\end{proposition}
		\begin{proof}
			Let $g\in {G_{\textrm{Int}}}\without{1}$ be such that $\norm{g}_S\leq n$.
			We assume throughout the proof that $n$ is large, that is in the proof we make a finite number of estimates that each hold for all $n\geq N_0$ for some fixed integer $N_0$. This is sufficient as if $\norm{g}<N_0$, then we still obtain that $\rf_{{G_{\textrm{Int}}},S}(\norm{g})\leq \rf_{{G_{\textrm{Int}}},S}(N_0)$.

			By \Cref{prop:rflamplighter} we may assume that $g\in C$ and thus by $\Cref{prop:supportC}$ we may find $\gamma_i$ for $i\in[1,n]$ with $\abs{\gamma_i}\leq n^2$ such that $g=\sum \gamma_ic_i$ (where we switch to additive notation in the abelian group $C$).
			Let $S(g)=\{i\in[1,n]\mid \gamma_i\neq 0\}$ we distinguish $2$ cases.
			
			First assume that $S(g)\subset\{d_i\mid i\in\ZZ_{\geq1}\}$.
			Let then $d_k$ be minimal such that $\gamma_{d_k}\nmid q_k$. Such $k$ always exists as otherwise $g$ would be trivial. Let now $p\mid q_k$ be minimal such that $p\nmid \gamma_{d_k}$. Such a prime exists as $q_k$ is square free. Either this prime is given by $P_{5k-5}<q_k<n$, the smallest prime divisor of $q_k$ or there exists some maximal $j$ such that $P_j\mid q_k,P_j\mid \gamma_{d_k}$. In this case $P_{j+2}$, which for large $n$ is bounded by $2{\mf f}^4(P_j)\leq 2{\mf f}^4(n^2)$, is not a divisor of $\gamma_{d_k}$. In both cases, $p$ is less than ${\mf f}^5(n)$.
			We claim in this case that $\varphi_{p,2^{k+1}}(g)\neq 0$. Indeed notice $2^{i+1}\mid d_{j}$ for all $j>k$. In particular, $\varphi(\sum \gamma_ic_i)=\gamma_{d_k}c_{d_k}$ which by \Cref{prop:QuotientOfC} is non-trivial.
			We thus have in this case by \Cref{prop:order} that $g$ does not vanish in some quotient of size at most ${2^{k+2}}p^{2^{k+3}}$, which by \Cref{prop:bound2power} is at most $4\log(n){\mf f}^5(n)^{8\log(n)}$ for $n$ sufficiently large.
			In this case, the bound follows.
						
			Now assume that $S(g)\not\subset\{d_i\mid i\in\ZZ_{\geq1}\}$.
			Let $f\in \ZZ[X]$ be the polynomial $$f(X)=\sum_{i\in[1,n]}\gamma_i X^{i}$$
			and let $\tilde f\in\ZZ[X,X\inv]$ be the Laurent-polynomial $$
			\tilde f(X)=\sum_{i\in[1,n]}\gamma_i (X^{i}-X^{-i}).
			$$
			Let $k$ be minimal such that $n<d_k$ and let $l_0,l_1,l_{-1},\cdots l_{k-1},l_{-k+1}$ be given respectively by
			${0,d_1,-d_1,d_2,-d_2,\cdots d_{k-1},-d_{k-1}}$. By looking at the number of factors $2$, notice that all the $l_i$ are distinct $\mod 2^{k+1}$.
			By \Cref{prop:laurentCenterReduct}, for $n$ sufficiently large, we find some odd integer $q\leq c(2n+1)^{\frac{1+\epsilon}{2}}$ such that $\tilde f\notin I_q+M$ for some uniform constant $c$ where 
			$$M=\ZZ X^{0}+\ZZ X^{d_1}+\ZZ X^{-d_{1}}+\cdots+ \ZZ X^{d_{k-1}}+ \ZZ X^{-d_{k-1}}$$ and $I_q$ is the ideal $\langle X^{2^{k+1}q}-1\rangle$. 
			We rewrite $\tilde f+I_q$ as $$
			\sum_{i\in\frac{\ZZ}{2^{k+1}q\ZZ}}\overline\gamma_iX^{i}\in \frac{\ZZ[X,X\inv]}{I_q}.
			$$
			We find some non-zero $\overline\gamma_m$ with $m\ncong 0,\pm d_1,\cdots \pm d_{k-1}\mod 2^{k+1}q$. Notice that $\overline\gamma_m$ is given by $$
			\sum_{\substack{j=m+i2^{k+1}q,i\in\ZZ\\ j\in[-n,n]}}\gamma_j
			$$
			where we use the convention that $\gamma_{-j}=-\gamma_j$.
			We thus have that $\overline\gamma_m$ is bounded above by $2n^3$.
			Let $P_j$ be the least divisor of $q_k$ such that $P_j>\abs{\overline\gamma_m}$.
			Notice that $P_{5k+11}>2d_k^3$ is a divisor of $q_k$ thus such a divisor exists. Furthermore, if $n$ is large, then this divisor is bounded above by $2{\mf f}^4(2n^3)$, indeed $P_{5k-5}$, the least prime divisor of $q_k$, is bounded above by  $2{\mf f}(d_{k-1})\leq 2\mf f(2n)$ and the prime divisors of $q_k$ lie at most $2{\mf f}^4$ apart.
			As $P_j>\abs{\overline\gamma_m}$, we have that $\tilde f\notin M+I_q+P_j\ZZ[X,X\inv]$.
			
			Consider the following morphisms of abelian groups: $$\pi:\ZZ[x]\rightarrow C:X^{i}\mapsto c_i,$$ $$\delta:\ZZ[x]\rightarrow\ZZ[X,X\inv]:X^{i}\mapsto X^{i}-X^{-i}$$and $$
			\pi_{P_j,2^{k+1}q}:\ZZ[X,X\inv]\rightarrow \tilde C_{p,2^{k+1}q}:X^{i}\mapsto c_{\overline i}
			$$
			where $\overline i\in[-2^{k+1}q,2^{k+1}q]$ is such that $i\cong\overline i\mod2^{k+2}q$
			Notice that given these morphisms, the following diagram commutes.
			
			\[\begin{tikzcd}
				{\ZZ[X]} && {\ZZ[X,X\inv]} \\
				C & {C_{p,2^{k+1}q}} & {\tilde C_{p,2^{k+1}q}}
				\arrow["\delta", from=1-1, to=1-3]
				\arrow["\pi"', from=1-1, to=2-1]
				\arrow["{\pi_{P_j,2^{k+1}q}}"', from=1-3, to=2-3]
				\arrow["{\varphi_{P_j,2^{k+1}q}}"', from=2-1, to=2-2]
				\arrow["\delta"', from=2-2, to=2-3]
			\end{tikzcd}\]
			
			Notice that the kernel of $\pi_{P_j,2^{k+1}q}$ is contained in $M+I_q+P_j\ZZ[X]$. In particular we have $\pi_{P_j,2^{k+1}q}(\tilde f)\neq 0$ or thus $(\pi_{P_j,2^{k+1}q}\comp\delta)(f)\neq 0$. By commutativity, we have that ${(\delta\comp\varphi_{P_j,2^{k+1}q}\comp\pi)(f)}$ is non-trivial and thus is $\varphi_{P_j,2^{k+1}q}(\pi(f))=\varphi_{P_j,2^{k+1}q}(g)$ non-trivial.
			We thus have that $g$ does not vanish in the quotient $G_{p,2^{k+1}q}$ which by \Cref{prop:order} is of order at most $2^{k+2}q(P_j)^{2^{k+3}q}$. By the above, and by \Cref{prop:bound2power}, this can be bounded from above for large $n$ by$$
						4c\log(n)(2n+1)^{\frac{1+\epsilon}{2}}(2{\mf f}^4(2n^3))^{8c\log(n)(2n+1)^{\frac{1+\epsilon}{2}}}.
			$$
			Again the upper bound follows.

		\end{proof}
		Combining the results of this section we obtain the following for ${G_{\textrm{Int}}}$
		\setcounter{result}{3}
		\begin{result}\label{prop:ExistsIntermediate}
			There exists a finitely generated recursively presented $3$-step solvable group with intermediate residual finiteness growth.
		\end{result}
		One wonders if the same holds for the conjugacy separability, that is
		\begin{question}
			Does there exist a finitely generated recursively presented group with intermediate conjugacy separability growth?
		\end{question}
		One might expect that finding such a group, if it exists, is significantly more difficult. As far as the author is aware, all groups for which the conjugacy separability growth is known to be sub-exponential are virtually nilpotent, and by \cite{dere2019effective} these groups have polynomial conjugacy separability growth.

		\bibliographystyle{plain}
		\bibliography{bibliography}
\end{document}